\documentclass[11pt, letterpaper]{amsart}

\calclayout

\usepackage{graphicx}
\usepackage{amssymb}
\usepackage{amsmath}
\usepackage[dvipsnames]{xcolor}
\usepackage{commath}
\usepackage{amsfonts}%
\usepackage{graphicx,verbatim}
\usepackage{tikz}
\usetikzlibrary{cd}

\usepackage{mathtools} 

\usepackage{amsthm}

\usepackage{units}

\newcommand{\fp}{\mathfrak{p}}
\newcommand{\fq}{\mathfrak{q}}
\newcommand{\fd}{\mathfrak{d}}

\newcommand{\fP}{\mathfrak{P}}
\newcommand{\fD}{\mathfrak{D}}

\newcommand{\cO}{\mathcal{O}}

\newcommand{\val}{\mbox{val}}

\newcommand{\valfq}{\mbox{val}_{\fq}}
\newcommand{\valpi}{\mbox{val}_{\pi}}

\newcommand{\Disc}{\mbox{Disc}}

\newtheorem{theorem}{Theorem}
\newtheorem{definition}{Definition}
\newtheorem{proposition}{Proposition}
\newtheorem{corollary}{Corollary}
\newtheorem{lemma}{Lemma}
\newtheorem{conjecture}{Conjecture}
\newtheorem{remark}{Remark}

\usepackage{lineno}

\begin{document}

\title[]{The Multiplicative Formula of Langlands for Orbital Integrals in $GL(2)$}
\author{Malors Espinosa}
\address{Department of Mathematics, University of Toronto
Room 6290, 40 St. George Street, Toronto, Ontario, Canada M5S 2E4}
\email{srolam.espinosalara@mail.utoronto.ca}
\date{\today}
\maketitle
\begin{abstract}
In \cite{LanBE04}, Langlands introduces a formula for a specific product of orbital integrals in $\mbox{GL}(2, \mathbb{Q})$. Altu\u{g}, in \cite{AliI}, employs this formula to manipulate the regular elliptic part of the trace formula, with the aim of eliminating the contribution of the trivial representation from the spectral side.

In \cite{ARTHUR2018425}, Arthur predicts that this formula coincides with a product of polynomials associated with zeta functions of orders developed by Zhiwei Yun in \cite{ZYun}. In \cite{malors21-2}, we determined the explicit polynomials for the relevant quadratic orders. This paper demonstrates how these polynomials can effectively generalize Langlands' formula to $GL(2, K)$, for general algebraic number fields $K$.

Furthermore, we also use this formula to extend a well-known formula of Zagier to any algebraic number field and explain its applications in the contexts of the strategy of Beyond Endoscopy proposed by Langlands.

\end{abstract}


\section{Introduction and Motivation}

\subsection{Context of the problem}

The main conjecture of the Langlands Program is the Principle of Functoriality. Despite the  progress that has been achieved in the study of this problem, and the implications it has in number theory, representation theory and the theory of the trace formula, it has not been settled in full generality. In \cite{LanBE04}, Langlands proposes a strategy to prove this conjecture which is now known as the strategy of Beyond Endoscopy or simply BE. It requires the construction of a new trace formula in which the poles of automorphic $L-$functions play a prominent role. 

The present article is motivated by precise questions in BE which we discuss below. To see the concrete nature of the problems we will consider, as well as the motivation and importance they have, we must briefly review the strategy of BE.

Recall that a trace formula is an identity equating two distributions built from different information of a reductive algebraic group, for example $G = \mbox{GL}(n)$, which is the only case we will consider below. Each side of the trace formula has its main part. For the geometric side, the main part is the regular elliptic part 
\begin{equation*}
    \displaystyle\sum_{[\gamma]}\mbox{vol}(\gamma) \mathcal{O}(\gamma, f),
\end{equation*}
where $[\gamma]$ runs over all regular elliptic conjugacy classes in $GL(n, \mathbb{Q})$ and $f$ is a test function.  Furthermore, $\mbox{vol}(\gamma)$ refers to the adelic volume of the centralizer $G_{\gamma}$ of $\gamma$ and $\cO(\gamma, f)$ is the orbital integral of $f$. 

On the spectral side, the main term is the discrete part which is \begin{equation}\label{spectralside}
    \displaystyle\sum_{\pi}m_{\pi}\mbox{Tr}(\pi(f)).
\end{equation}
It is the part of the spectral side of the trace formula that decomposes as a discrete sum. Here  $f$ is again a test function and $\pi$ runs over the representations supporting the discrete part.

The strategy of BE postulates that there must be a \emph{new} trace formula whose corresponding discrete part is
\begin{equation}\label{BEspectralside}
\displaystyle\sum_{\pi}\mu_{\pi, r}m_{\pi}\mbox{Tr}(\pi(f)),
\end{equation}
where $\mu_{\pi, r}$ is the order of the pole at $s = 1$ for the automorphic $L-$function $L(s, \pi, r)$ associated to $\pi$ and to some irreducible representation $r$ of $GL(n, \mathbb{C})$. This new trace formula changes as we vary the representation $r$ (\textit{i.e.} we would have an $r-$trace formula for each $r$). The objective is to exploit this variation to classify representations, since proper functorial transfers of representations should be those with poles for some $r$.

Nevertheless, that these poles exist and behave in the desired way to make this definition valid is a consequence of Functoriality and cannot be taken for granted. Instead, we suppose it and work backwards. We do this for two reasons: firstly, we want an expression that does not rely on Functoriality for its definition. Secondly, it is not clear how the main geometric part (\textit{i.e.} the regular elliptic part) should be modified to become the dominant part of the geometric side of this conjectural $r-$trace formula. We expect to obtain an expression for this out of this process. 

For the manipulations that follow, Langlands supposes $\pi$ to vary only in the representations that are unramified at all finite places and which, consequently, have trivial central character. Tauberian limit theory then implies that, supposing Functoriality, we have
\begin{equation}\label{tauberianlimit}
    \mu_{\pi, r} = \lim_{n\rightarrow\infty}\dfrac{1}{\sharp\{ p\mid p\le n\}}\displaystyle\sum_{p\le n} \log(p)\mbox{Tr}(r(c(\pi_p)))
\end{equation}
where the sum varies over primes and $c(\pi_p)$ is the semisimple conjugacy class associated to $\pi_p$ where $\pi = \bigotimes_p \pi_p$ is the decomposition of $\pi$ into local representations. Substituting this back into equation \eqref{BEspectralside}, and exchanging the order of summation, we get
\begin{equation}\label{manipulation1}
    \lim_{n\rightarrow\infty}\dfrac{1}{\sharp\{ p\mid p\le n\}}\displaystyle\sum_{p\le n}\log(p)\displaystyle\sum_{\pi} m_{\pi}\mbox{Tr}(r(c(\pi_p)))\mbox{Tr}(\pi(f)).
\end{equation}
If we now suppose our test function $f$ factorizes, that is, it can be written as
\begin{equation*}
    f = \displaystyle\prod_{p < \infty}f_p \times f_{\infty}, 
\end{equation*}
then we can construct for each fixed prime $q$ a smooth function $h_q$ of $GL(2, \mathbb{Q}_q)$ such that the new function $f^q$ defined by
\begin{equation*}
    f^q = \displaystyle\prod_{p < \infty, p\neq q}f_p \times (h_qf_q)\times f_{\infty}, 
\end{equation*}
satisfies
\begin{equation*}
    \mbox{Tr}(r(c(\pi_q)))\mbox{Tr}(\pi(f)) = \mbox{Tr}(\pi(f^q)).
\end{equation*}
Concretely, $h_q$ is the unramified spherical function on $GL(2, \mathbb{Q}_q)$ such that
\begin{equation*}
    \mathcal{S}(h_q)(c_q) = \mbox{Tr}(r(c_q)),
\end{equation*}
for any Satake parameter $c_q$ in $GL(2, \mathbb{Q}_q)$ and where $\mathcal{S}$ is the Satake transform.
Using such functions, as we vary the prime $p$ in the sum \eqref{manipulation1}, we get that the discrete part of the $r$-trace formula should be
\begin{equation}\label{manipulation2}
    \lim_{n\rightarrow\infty}\dfrac{1}{\sharp\{ p\mid p\le n\}}\displaystyle\sum_{p\le n}\log(p)\displaystyle\sum_{\pi} m_{\pi}\mbox{Tr}(\pi(f^p)).
\end{equation}
Notice that the inner sum of equation \eqref{manipulation2} above is the discrete spectral side of a classical trace formula for the test function is $f^p$. With this observation we realize that the $r$-trace formula must be a weighted average of classical trace formulas. We remark that this expression does not depend on Functoriality to be defined.

The main parts of each side of the trace formula are not equal as distributions. However, for purposes of motivation we can treat them as if they were. In such case, the regular elliptic part of the $r-$trace formula must be
\begin{equation}\label{manipulation3}
    \lim_{n\rightarrow\infty}\dfrac{1}{\sharp\{ p\mid p\le n\}}\displaystyle\sum_{p\le n}\log(p)\displaystyle\sum_{[\gamma]} \mbox{vol}(\gamma)\cO(\gamma, f^p).
\end{equation}
With this, we have obtained a candidate for the main term of the geometric side of the  $r$-trace formula. 

At this stage, the supposed equality between the regular elliptic part and discrete part of the desired $r-$trace formula manifests in the equality of the previous two limits in equations \eqref{manipulation2} and \eqref{manipulation3}. Both make sense by themselves without Functoriality but present the serious analytic issue of whether they converge. Due to the presence of nontempered representations the spectral side limit will not converge. In the situation where the $\pi$ are assumed to be unramified at all finite places, the only such nontempered representation is the trivial one. Hence, we could isolate it from the sum and merge it with the geometric side limit (\textit{i.e.} pass it to the other side). We get
the equality between
\begin{equation}\label{manipulation4}
     \lim_{n\rightarrow\infty}\dfrac{1}{\sharp\{ p\mid p\le n\}}\displaystyle\sum_{p\le n}\log(p)\left(\displaystyle\sum_{[\gamma]} \mbox{vol}(\gamma)\cO(\gamma, f^p) - \mbox{Tr}(\mathbf{1}(f^p))\right)
\end{equation}
and
\begin{equation*}\label{manipulation5}
     \lim_{n\rightarrow\infty}\dfrac{1}{\sharp\{ p\mid p\le n\}}\displaystyle\sum_{p\le n}\log(p)\displaystyle\sum_{\pi\neq \mathbf{1}} m_{\pi}\mbox{Tr}(\pi(f^p)) 
\end{equation*}
We expect that what remains in the discrete spectral side converges since we have taken away the nontempered terms, so what is on the geometric side \textit{must} converge. Hence, the following is unavoidable: \textit{it should be possible to isolate the trace of the trivial representation and cancel its contribution from the spectral side, so that what remains when plugged into the limit converges. Furthermore, both geometric and spectral limits must converge to the same value.} 

Thus, for the strategy of BE to hold, the first problem of interest is the isolation of the nontempered spectrum from the regular elliptic part. In \cite{frenkel2010formule} the idea that the isolation process can be achieved by applying an additive Poisson Summation was introduced. This is exactly what Altu\u{g} did in \cite{AliI} in the case where $f_p$ is the indicator of the standard maximal open compact subgroup at each finite prime $p$ and $r$ is the standard representation.

In \cite{AliI}, Altu\u{g} works with $\mbox{GL}(2, \mathbb{Q})$. He also takes $r$ to be the $k$-th symmetric power of the standard representation. Concretely, in \cite{AliI}, he gives an expression for 
\begin{equation*}
    \displaystyle\sum_{[\gamma]}\mbox{vol}(\gamma)\cO(\gamma, f^p) - \mbox{Tr}(\mathbf{1}(f^p)),
\end{equation*}
in terms of the trace and the determinant of $[\gamma]$. In order to perform the Poisson Summation required, and then the isolation of the trace of the trivial representation, two steps are introduced in \cite{AliI} and that are of relevance for any future work that plans to follow these steps.

The first one is the application of the approximate functional equation to a zeta function which has been introduced by Zagier in \cite{zagier}. This step is necessary to smooth the singularities of orbital integrals so that the hypothesis necessary for Poisson summation are satisfied. The second one is the evaluation of some Kloosterman sums, that appear once the Poisson summation has been performed, and that have very specific values. It is these values which are responsible for the isolation of the trace of the trivial representation. Both of these steps build on work that Langlands introduced in \cite{LanBE04}. Concretely, he gives an explicit formula for the product of the orbital integrals over the nonarchimidean primes. This formula turns out to be the value at $s = 1$ of the aforementioned formula of Zagier. This realization is what drives the whole work of Altu\u{g} in \cite{AliI}. 

Finally, in \cite{AliII} and \cite{AliIII} it is proven that the limit in equation \eqref{manipulation4} converges to the correct value. 

\subsection{The work of this paper}

It is of interest to generalize the work of Altu\u{g} to other contexts. One of the most natural directions to extend these results is to change the base field, that is, to work over an arbitrary number field $K$ as opposed to $\mathbb{Q}$. This is being carried out in the article \cite{BEgenfields}. In this paper, our work will be in the efforts of making this generalization possible.

Our main goal in this paper is to generalize the zeta function of Zagier, that Altu\u{g} uses in \cite{AliI} over the rational numbers, to an arbitrary number field. In the process of doing so, we will also prove the analogous formula of Langlands for the product of local orbital integrals. 

These two perspectives, the zeta function of Zagier and the orbital integrals, are connected by the work of Yun in \cite{ZYun}. Concretely, he constructs zeta functions of certain orders and proves that at $s = 1$ they recover the values of the orbital integrals. This lead Arthur to conjecture in \cite{ARTHUR2018425} that the zeta functions of these orders are the nonarchimedean Euler factors of Zagier's zeta function.  We will prove that this is indeed the case.

Let $K$ be an algebraic number field and $\cO_K$ its ring of integers. Also let $\fq$ be any prime of $K$ of norm $q$. Denote by $K_{\fq}$ the local field obtained by completing $K$ with respect to its $\fq$-adic norm. Let $\cO_{K_\fq}$ be the ring of integers of $K_\fq$. We will use this setup to explain two different objects that we will use to state our main result.

For a given prime $\fq$ of $K$ and a nonnegative integer $k$ we define the function $f_{\fq, k}$ on $\mbox{GL}(2, K_\fq)$ as
\begin{equation}\label{localtestfn}
    f_{\fq, k} = \mathbf{1}\left(X\in\mbox{Mat}(2, \cO_{K_\fq}) \mid\; |\det(X)|_\fq = q^{-k}\right).
\end{equation}
Let $\gamma$ be a regular elliptic element in $\mbox{GL}(2, K_\fq)$. Associated to $\gamma$ and $f_{\fq, k}$ there is an associated orbital integral $\cO(\gamma, f_{\fq, k})$ (we refer the reader to section \ref{Multiplicative Formula of Langlands} below for the precise definition).

It follows easily from the definition of $f_{\fq, k}$ that
\begin{equation*}
    \cO(\gamma, f_{\fq, k}) \neq 0
\end{equation*}
only if $|\det(\gamma)|_\fq = q^{-k}$. We refer to such a matrix a 
\emph{contributing} matrix of $f_{\fq, k}$. 

We can enhance the previous local construction to a global one as follows. Let $k_\fq$ be a nonnegative integer for each prime $\fq$ of $K$ with the assumption that $k_\fq = 0$ for all but finitely many $\fq$. Using these $k_\fq$ we can define
\begin{equation*}
    f = \displaystyle\prod_\fq f_{\fq, k_\fq}.
\end{equation*}
We do not put the dependence of $f$ on the sequence of $k_\fq$ to ease notation. It will always be clear from the context which sequence is being used. Furthermore, using the embedding of $K$ into its completion $K_\fq$ we can consider any matrix $\gamma\in\mbox{GL}(2, K)$ as an element of $\mbox{GL}(2, K_\fq)$. Doing this we can define
\begin{equation*}
    \cO(\gamma, f) := \displaystyle\prod_\fq \cO(\gamma, f_{\fq, k_\fq}).
\end{equation*}
This is of course an orbital integral on the finite part of the adeles.

In another direction, associated to a global regular elliptic $\gamma$ there is the quadratic extension $K_{\gamma} = K(\gamma)$, generated by adjoining a root of the characteristic polynomial of $\gamma$ to $K$. We know we can identify this root with $\gamma$ and we do so from now on. $K_{\gamma}/K$ is a quadratic extension of number fields. We can also repeat this in the local case, that is, generate $K_\fq(\gamma) := K_{\gamma}\otimes_K K_\fq$. Nevertheless, according to the splitting behavior of $\fq$ in $K_{\gamma}$, this extension can be one of three options: ramified quadratic field extension, unramified quadratic field extension or $K_\fq \times K_\fq$, when the extension splits.  

For a $\gamma$ with $\tau$ and $\det(\gamma)$ integral, we can construct in each of these situations, i.e. local or global, the monogenic orders
\begin{equation*}
    \mathcal{O}_{K}[\gamma] = \{a + b\gamma \mid a, b\in \mathcal{O}_{K}\},\;\;
    \mathcal{O}_{K_\fq}[\gamma] = \{a + b\gamma \mid a, b\in \mathcal{O}_{K_\fq}\}.
\end{equation*}
These are subrings of the corresponding ring of integers of the extension, but do not necessarily coincide with it. 

In \cite{ZYun}, Yun defines a zeta function for general orders, which for the case of monogenic ones, is simply the Dirichlet series that counts ideals. That is,
if $\cO$ is a monogenic order his definition coincides with
\begin{equation*}
    \zeta_{\cO}(s) = \displaystyle\sum_{I\subseteq\cO} \dfrac{1}{[\cO : I]^s},
\end{equation*}
where the sum runs over ideals whose index in the order is finite. Yun proves in the local case that the zeta function is rational and can be written as
\begin{equation*}
    \zeta_{\cO}(s) = \dfrac{P(q^{-s})}{V(s)},
\end{equation*}
where both $V(s)$ and $P(q^{-s})$ are polynomials in $q^{-s}$.   
In \cite{ZYun} it is proven that, for a certain specific integer $n$, the function
\begin{equation*}
    \tilde{J}(s) = q^{ns}V(s)\zeta_{\cO}(s)
\end{equation*}
satisfies the functional equation
\begin{equation*}
    \tilde{J}(s) = \tilde{J}(1 - s).
\end{equation*}
We will discuss these results more carefully in section \ref{Explicit Zeta Functions of Certain Orders} below. 

In particular, returning to our setup, for each prime $\fq$ and local order $\mathcal{O}_{K_\fq}[\gamma]$, we denote the corresponding function $\Tilde{J}$ by
\begin{equation*}
    \tilde{J}_{O_{K_\fq}[\gamma]}(s).
\end{equation*}
Our first main result is the following result: 

\begin{theorem}(Theorem \ref{generalformula} in section \ref{Multiplicative Formula of Langlands})\label{generalformulaintro}
 Let $k_\fq$ be a nonnegative integer for each prime $\fq$ of $K$ with $k_\fq = 0$ for all but finitely many $\fq$. Let
\begin{equation*}
    f = \displaystyle\prod_\fq f_{\fq, k_\fq},
\end{equation*}
and $\gamma$ a contributing matrix for this $f$.  
We have the expansion
\begin{equation*}
    \cO(s, \gamma, f) := \displaystyle\prod_{\fq} \tilde{J}_{O_{K_\fq}[\gamma]}(s) = N_K(S_{\gamma})^s\displaystyle\sum_{\fd|S_{\gamma}}N_K(\fd)^{(1 - 2s)}\displaystyle\prod_{\fp|\fd'}\left(1 - \dfrac{\chi_{\gamma}(\fp)}{N_K(\fp)^{s}}\right), 
\end{equation*}
where $\fd' = \dfrac{S_{\gamma}}{\fd}$. It is entire and satisfies the functional equation
\begin{equation*}
    \cO(1-s, \gamma, f) = \cO(s, \gamma, f).
\end{equation*}
In particular, evaluating at $s = 1$, we get
\begin{equation*}
    \displaystyle\prod_\fq \cO(\gamma, f_{\fq, k}) = \displaystyle\prod_{\fq|S_{\gamma}} \cO(\gamma, f_{\fq, k}) = \displaystyle\sum_{\fd|S_{\gamma}}N_K(\fd)\displaystyle\prod_{\fp|\fd}\left(1 - \dfrac{\chi_{\gamma}(\fp)}{N_K(\fp)}\right)
\end{equation*}
\end{theorem}
In the previous proposition $\chi_{\gamma}$ is the quadratic Hecke character of the extension $K_{\gamma}/K$ and $S_{\gamma}$ is a certain ideal of $\cO_K$.
Furthermore, the evaluation at $s = 1$ follows from the fact
\begin{equation*}\label{orbitzetaintro}
    \tilde{J}_{O_{K_\fq}[\gamma]}(1) = \cO(f_{p, k}, \gamma),
\end{equation*}
which Yun also proves in \cite{ZYun}. 

We are interested in this equation because it is the one that manifests the link between the orbital integrals (and the theory of the trace formula) with zeta functions of orders and their analytic properties. 


As we have explained above, we have a decomposition for each prime
\begin{equation*}
    \tilde{J}_{O_{K_\fq}[\gamma]}(s) = q^{n_{\fq}s}P_{O_{K_\fq}[\gamma]}(q^{-s}),
\end{equation*}
where $P_{O_{K_\fq}[\gamma]}$ is the polynomial in the numerator of the rational expression of $ \zeta_{O_{K_p}[\gamma]}(s)$ and $n_{\fq}$ is a nonnegative integer that depends on $\gamma$ and $\fq$. 

In \cite{malors21-2} the author has proved that in the case of monogenic orders $O$ inside the integral elements of  a quadratic reduced algebra extension of a local field, we have the polynomials $P_{O}$ split in three families according to whether the extension is a ramified field extension, an unramified field extension or a split extension. These polynomials are, respectively, 
\begin{align*}
    R_n(X) &= 1 + qX^2 + q^2X^4 + ... + q^nX^{2n},\\
    U_n(X) &= 1 + X + qX^2 + qX^3 + ... + q^{n - 1}X^{2n - 2} + q^{n - 1}X^{2n - 1} + q^nX^{2n},\\
    S_n(X) &= 1 - X + qX^2 - qX^3 + ... + q^{n - 1}X^{2n - 2} - q^{n - 1}X^{2n - 1} + q^nX^{2n}.
\end{align*}
Even though these polynomials are related, they are different and cannot be described in a uniform way. Nevertheless in the context where everything is constructed out of a regular elliptic matrix the quadratic Hecke character $\chi_{\gamma}$ associated to the quadratic extension $K_{\gamma}/K$ allows us to give a uniform expression. We get the general formula
\begin{equation*}
     \tilde{J}_{O_{K_\fq}[\gamma]}(s) = q^{n_\fq s}\left(\left(1 - \dfrac{\chi_{\gamma}(\fq)}{q^{s}}\right)R_{n_\fq-1}(s) + q^{n_\fq}q^{-2n_\fq s}\right).
\end{equation*}
It is this general expression which allow us to prove the result.  

Langlands considers over $\mathbb{Q}$ the choice $k_p = 0$ for all primes except one prime $p_0$ for which $k_{p_0}$ is an arbitrary nonnegative integer. Then he gets, as formula (59) in \cite{LanBE04}, and for the corresponding function $f$ the following
\begin{proposition}\label{langlandsformulaprop}
For a contributing matrix $\gamma$ for $f$ there exists a positive integer $S_{\gamma}$ such that
\begin{equation*}
    \tau^2 \mp 4p_0^{k_{p_0}} = S_{\gamma}^2D_{\gamma}
\end{equation*}
where the characteristic polynomial of $\gamma$ is $X^2 - \tau X \pm p_0^{k_{p_0}}$. Furthermore, we have
\begin{equation}\label{langlandsformulaintro}
    \displaystyle\prod_{q<\infty} \cO(\gamma, f_q) = \displaystyle\sum_{d|S_{\gamma}} d\displaystyle\prod_{p|d}\left(1 -  \dfrac{\left(\frac{D_{\gamma}}{p}\right)}{p}\right),
\end{equation}
where $D_{\gamma}$ is the fundamental discriminant of the extension $\mathbb{Q}(\gamma)$. In the above equation the product runs over all primes $p$ dividing $d$. Here $\left(\frac{D_{\gamma}}{\cdot}\right)$ is the Kronecker Symbol.
\end{proposition}

We can define the complex variable function 
\begin{equation*}
S_{\gamma}^s\displaystyle\sum_{d|S_{\gamma}}d^{(1 - 2s)}\displaystyle\prod_{p|d'}\left(1 - \dfrac{\left(\frac{D_{\gamma}}{p}\right)}{p^{s}}\right), 
\end{equation*}
where $d' := S_{\gamma}/d$. If we evaluate this function at $s = 1$ we recover precisely the right hand side of equation \eqref{langlandsformulaintro} above. Arthur's prediction is the following
\begin{conjecture}\label{arthurconjecture}
The following equality holds
\begin{equation*}
S_{\gamma}^s\displaystyle\sum_{d|S_{\gamma}}d^{(1 - 2s)}\displaystyle\prod_{p|d'}\left(1 - \dfrac{\left(\frac{D_{\gamma}}{p}\right)}{p^{s}}\right) = \displaystyle\prod_{q} \tilde{J}_{\mathbb{Z}_q[\gamma]}(s),
\end{equation*}
where $\tilde{J}_{\mathbb{Z}_q[\gamma]}(s)$ are as defined above when $K = \mathbb{Q}$.
\end{conjecture}
This conjecture is implicit in the discussion of \cite[section 3]{ARTHUR2018425}, especially the last paragraphs in page 446. In particular, he mentions he has verified it for $k_{p_0} = 0$. Using theorem \ref{generalformulaintro} above for the corresponding choice of $f$ the conjecture follows immediately. Furthermore, the analogous statement holds for all number fields.

We now discuss the formula of Zagier. For a nonsquare integer $\delta$ that is congruent to $0$ or $1$ modulo $4$ (the so called \emph{discriminants}) we define for a complex variable $s$ with $\Re(s)>1$
\begin{equation}\label{primegeodesic}
    L(s, \delta) = \displaystyle\sum_{f^2|\; \delta}'\dfrac{1}{f^{2s - 1}}L\left(s, \left(\dfrac{\frac{\delta}{f^2}}{\cdot}\right)\right).
\end{equation}
The $'$ in the sum means that it is taken only over $f$ such that $\delta/f^2$ is itself congruent to $0$ or $1$ modulo $4$. We call this the \emph{congruence conditions} and we refer to this function as Zagier's zeta function for $\mathbb{Q}$.  The case where $\delta$ is a square also makes sense, but the results go in another direction as all the $L-$ functions involved reduce to the Riemann zeta function and that changes the study of its analytic properties. Zagier's zeta function can be completed to an entire function $\Lambda(s, \delta)$ (this requires $\delta$ not being a square) that satisfies the functional equation
\begin{equation*}
    \Lambda(s, \delta) = \Lambda(1 - s, \delta).
\end{equation*}

We will generalize this zeta function as follows: given a nonsquare algebraic integer $\delta\in\cO_K$ and an ideal $I$ of $\cO_K$, we say $I$ \textbf{satisfies the congruence conditions for $\delta$} if
 \begin{enumerate}
     \item $I^2$ divides $\delta$,
     \item for each prime $\fq$ of $\cO_K$ above $2$ we have
       \begin{equation*}
           \dfrac{\delta}{\pi_{\fq}^{2\valfq(I)}}
       \end{equation*} 
       is a square modulo $4\cO_{K_\fq}$, where $\pi_\fq$ is a uniformizer of $K_\fq$.  
 \end{enumerate} 
Then we define
\begin{equation*}
L(s, \delta) := \displaystyle\sum_{I^2|\delta}'\dfrac{1}{N_K(I)^{2s - 1}}L\left(s, \chi_{I}\right).    
\end{equation*}
Here $'$ means that the sum runs over the ideals that satisfy the congruence conditions for $\delta$. In the above definition, $\chi_{\delta}$ is the quadratic character of the extension $K(\sqrt{\delta})/K$ and $\chi_I$ is the character defined at prime ideals $\fq$ as
\begin{equation*}
    \chi_{I}(\fq) := \left\{
	\begin{array}{ll}
		\chi_{\delta}(\fq)   & \mbox{if } \fq \mbox{ does not divide } S_{\delta}/I \\
		0  & \mbox{if } \fq \mbox{ divides } S_{\delta}/I
	\end{array}
	\right.
\end{equation*}
Here $S_{\delta}$ is a certain ideal of $\cO_K$. To study the properties of this function, we will construct a regular elliptic matrix $\gamma\in GL(2, K)$ such that
\begin{equation*}
    \mbox{Trace}(\gamma)^2 - 4\det(\gamma) = \delta,
\end{equation*}
and prove
\begin{equation}\label{equalityLfns}
    L(s, \delta) = N_K(S_{\gamma})^{-s}L(s, \chi_{\gamma})\cO(s, \gamma, f).
\end{equation}
This will allow us to transfer to $L(s, \delta)$ the properties of $L(s, \chi_{\gamma})$, which is a standard Dirichlet $L-$function, from standard facts of number theory. It will also inherit those of $\cO(s, \gamma, f)$ from theorem \ref{generalformulaintro} above. From this we will prove that our generalization of Zagier's zeta function can be completed to an entire function which satisfies a functional equation. We will do this in section \ref{Generalization of Zagier's Zeta Function} below. Of course, choosing the base field to be the rational numbers recovers the well known properties of Zagier's Zeta function over $\mathbb{Q}$.

We emphasize an important point: notice that in the expansion of $\cO(s, \gamma, f)$ the indexing of the sum is over ideals $I$ such that $I|S_{\gamma}$. This is because the product of the zeta functions of orders naturally give this condition. Instead, the definition of our generalization of Zagier's zeta function uses the congruence conditions to index the sum. We will prove the equivalence of the indexing sets (see proposition \ref{prop: congruence conditions} in section \ref{The Congruence Conditions})

This concludes our discussion of the work and organization of this paper. We just add that in section \ref{Relevance of our work in the strategy of BE} we show how our work is used concretely in the strategy of BE, by discussing the case of $\mathbb{Q}$. Finally, in section \ref{Questions and future work} we discuss some questions that arise from our work.


\section{Explicit Zeta Functions of Certain Orders}\label{Explicit Zeta Functions of Certain Orders}

In the paper \cite{ZYun} a zeta function for general orders of number fields is defined. Furthermore, a local version of such zeta function is also defined for orders of reduced $K$- algebras where $K$ is a local field. This general version is necessary to account for the diversity of algebras that appear when considering extensions by polynomials since an irreducible polynomial over a global number field may not remain irreducible over the completion at a prime.

We begin our brief discussion of \cite{ZYun} with the local case.
\begin{definition}\label{zetafunction}
Let $K$ be a non-Archimedean local field of characteristic zero and $\cO\subseteq O_K$ be a monogenic order, where $O_K$ is the ring of integers of $K$. We define its zeta function as
\begin{equation*}
    \zeta_{\cO}(s) = \displaystyle\sum_{I\subseteq\cO} \dfrac{1}{[\cO : I]^s},
\end{equation*}
where the sum runs over all ideals $I\subseteq\cO$ of finite index. 
\end{definition}
In \cite{ZYun} a more general zeta function is defined, but it is not the one we have defined above. For general orders, and not only monogenic ones, a more  careful construction is needed in order to make sure the results proven hold true. In the case of the orders we will deal with the previous definition is correct and coincides with the zeta function defined in \cite{ZYun}.

The main result of \cite{ZYun} is the existence of a functional equation for the zeta functions of orders. The proof goes by studying them locally, and verifying that the local versions satisfy the functional equation when they are properly modified.

For now, we are interested in the local situation so we will restrict to it. In this case, we have a non-Archimedean local field of characteristic zero $K$ and a finite dimensional reduced $K-$algebra $L$ which can be written as product of $g$ fields
\begin{equation*}
    L = L_1\times...\times L_g. 
\end{equation*}
The residue field of each one of these fields has cardinality $Q_i = q^{f_i}$ and the prime ideal $\pi_iO_{L_i}$ in $O_{L_i}$ satisfies 
\begin{equation*}
    (\pi_i O_{L_i})^{e_i} = pO_{L_i},
\end{equation*}
where $\pi_i$ is a uniformizer of $O_{L_i}$ and $p$ is one of $O_K$. Here $e_i$ and $f_i$ are the ramification and inertia degrees of the extension $L_i/K_p$. With this at hand we define the following function of a complex parameter $s$
\begin{align*}
    V(s) &= \displaystyle\prod_{i = 1}^g(1 - q^{-f_is}),
\end{align*}
Considering now an order $\cO$ of $L$, define 
\begin{equation*}
    n = \mbox{length}_{O_K}(O_L/\cO) = \log_q([O_L: \cO]).
\end{equation*}
 Notice that $V(s)$ does not depend on the order $\cO$ selected. Now we can state the following proposition which is proved in \cite{ZYun} as theorem 2.5. 
\begin{proposition}\label{polynomial}
Let $K$ be a non-Archimedean local field of characteristic zero, $L$ be a finite dimensional reduced $K$-algebra and $\cO\subseteq O_L$.  In the context of the previous discussion, define the function
\begin{equation*}
    \tilde{J}(s) = q^{ns}V(s)\zeta_{\cO}(s).
\end{equation*}
Then there exists a polynomial $P(x)\in 1 + x\mathbb{Z}[x]$ of degree $2n$ such that
\begin{equation*}
    V(s)\zeta_{\cO}(s) = P(q^{-s}).
\end{equation*}
Furthermore, $\tilde{J}$ satisfies the functional equation
\begin{equation*}
    \tilde{J}(s) = \tilde{J}(1 - s),
\end{equation*}
or equivalently, the polynomial $P$ satisfies
\begin{equation*}
    (qx^2)^{n}P\left(\dfrac{1}{qx}\right) = P(x).
\end{equation*}
\end{proposition}

When $K$ is a non-Archimedean local field of characteristic zero field and $L$ is a quadratic reduced $K-$ algebra there are three possible scenarios: either $L/K$ is a field extension, which can be ramified or unramified, or $L = K\times K$. We refer to these cases as the ramified, unramified and split cases. In \cite{malors21-2} we have restricted to these three cases and computed explicitly what are the polynomials predicted by the previous proposition for certain family of orders. 

The setup is as follows: if $\cO_L$ is the integral closure of $\cO_K$ in $L$, we get that the standard theory of discrete valuation rings guarantees that there exists an integral element $\Delta\in\cO_L$ such that
\begin{equation*}
    \cO_L = \cO_K[\Delta].
\end{equation*}
We define a sequence of orders $\cO_0 \supset \cO_1 \supset \cO_2 \supset...$ by
\begin{equation*}
    \cO_n = \cO_K[p^n\Delta].
\end{equation*}
It is an easy exercise to verify that this sequence does not depend on the specific $\Delta$ chosen. Each of these orders has a zeta function as defined above. Let us call such zeta function $\zeta_n$.

With this notation now settled we can state the result:
\begin{theorem}\label{solutionrecurrence}
For each $n\ge 0$ define the following polynomials:
\begin{align*}
        R_n(X) &= 1 + qX^2 + q^2X^4 + ... + q^nX^{2n},
\end{align*}
and for $n\ge 1$ define
\begin{align*}
    U_n(X) &= (1 + X)R_{n - 1}(X) + q^nX^{2n},\\
    S_n(X) &= (1 - X)R_{n - 1}(X) + q^nX^{2n}.
\end{align*}
Finally, also put $U_0(X) = S_0(X) = 1$. Explicitly, these polynomials are
\begin{align*}
    R_n(X) &= 1 + qX^2 + q^2X^4 + ... + q^nX^{2n},\\
    U_n(X) &= 1 + X + qX^2 + qX^3 + ... + q^{n - 1}X^{2n - 2} + q^{n - 1}X^{2n - 1} + q^nX^{2n},\\
    S_n(X) &= 1 - X + qX^2 - qX^3 + ... + q^{n - 1}X^{2n - 2} - q^{n - 1}X^{2n - 1} + q^nX^{2n}.
\end{align*}Then we have 
\begin{align*}
    (1 - q^{-s})\zeta_n(s) &= R_n(q^{-s}),\\
    (1 - q^{-2s})\zeta_n(s) &= U_n(q^{-s}),\\
    (1 - q^{-s})^2\zeta_n(s) &= S_n(q^{-s}).
\end{align*}
That is, $R_n, U_n$ and $S_n$ are the polynomials predicted by proposition \ref{polynomial} in the ramified, unramified and split case, respectively.
\end{theorem}
We refer the reader to \cite{malors21-2} for a proof of this result.

Now let us discuss the number field case. There is also a definition for the global zeta function of a global order as follows

\begin{definition}
Let $K$ be a number field and $\cO\subseteq\cO_K$ be a monogenic order.  Define the zeta function of $\cO$ as
\begin{equation*}
    \zeta_{\cO}(s) = \displaystyle\sum_{I\subseteq\cO} \dfrac{1}{[\cO : I]^s},
\end{equation*}
and the completed zeta function as
\begin{equation*}
    \Lambda_{\cO}(s) = \mid\nicefrac{\cO_K}{\cO}\mid^s D_K^{s/2}(\pi^{-s/2}\Gamma(s/2))^{r_1}((2\pi)^{1 - s}\Gamma(s))^{r_2}\zeta_{\cO}(s),
\end{equation*}
where $D_K$ is the absolute value of the discriminant of $K$ and $r_1, r_2$ have their usual meaning as the number of real and complex places of $K$, respectively.
\end{definition}
Notice that 
\begin{equation*}
    D_K^{s/2}(\pi^{-s/2}\Gamma(s/2))^{r_1}((2\pi)^{1 - s}\Gamma(s))^{r_2}
\end{equation*}
are the usual $L-$factors at the infinite places used to complete the standard Dedekind zeta functions of $K$.

In \cite{ZYun} it is proven that $\Lambda_{\cO}(s)$ satisfies a functional equation 
\begin{equation*}
    \Lambda_{\cO}(s) = \Lambda_{\cO}(1 - s)
\end{equation*}
and has poles only at $s = 0, 1$ with residues given by certain expressions involving regulators, class numbers, etc. We refer to theorem 1.2 of \cite{ZYun} for a precise statement. Again, we have restricted to monogenic orders so that the expression we give above for the zeta function is the correct one but this can be done in much more generality.

What will matter for us is the following statement
\begin{proposition}\label{globalzeta}
Let $\cO\subseteq\cO_K$ be a monogenic order. Then we have
\begin{equation*}
    \dfrac{\Lambda_{\cO}(s)}{\Lambda_K(s)} = \displaystyle\prod_{\fp} \tilde{J}_{\cO_{\fp}}(s),
\end{equation*}
where $\Lambda_K(s)$ is the usual completed Dedekind zeta function of the number field $K$ and $\cO_{\fp}$ is the localization of the order $\cO$ at the prime $\fp$ of $\cO$. In here $\tilde{J}_{\cO_{\fp}}(s)$ is the function guaranteed by proposition \ref{polynomial} and the product runs over the prime ideals of $\cO$.
\end{proposition}

This is deduced as equation (3.2) of \cite{ZYun} within the proof of theorem $1.2$. We isolate it as a proposition since it will be necessary for one of our conclusions at a later section.


\section{Multiplicative Formula of Langlands}\label{Multiplicative Formula of Langlands}

From now on let $K$ be an algebraic number field. 
\begin{definition}
For a given prime $\fq$ of $K$ and a nonnegative integer $k$ we define a function on $\mbox{GL}(2, K_\fq)$ as
\begin{equation}\label{localtestfn}
    f_{\fq, k} = \mathbf{1}\left(X\in\mbox{Mat}(2, \cO_{K_\fq}) \mid\; |\det(X)|_\fq = q^{-k}\right).
\end{equation}
Given a nonnegative integer $k_\fq$ for each prime $\fq$ of $K$,  with the assumption that $k_\fq = 0$ for all but finitely many primes $\fq$, we define
\begin{equation*}
    f = \displaystyle\prod_\fq f_{\fq, k_\fq}.
\end{equation*}
\end{definition}
Let $\gamma$ be a regular elliptic element in $\mbox{GL}(2, K_\fq)$, that is, an invertible matrix whose characteristic polynomial is irreducible over $K_\fq$. Associated to $\gamma$ and $f_{\fq, k}$ there is an orbital integral defined as
\begin{equation}\label{padicorbitalintegral}
    \cO(\gamma, f_{\fq, k}) = \displaystyle\int_{\mbox{GL}_{\gamma}(2, K_\fq)\backslash \mbox{GL}(2, K_\fq)} f_{\fq, k}(x^{-1}\gamma x) \mbox{d}x_\fq,
\end{equation}
where $\mbox{d}x_\fq$ is the right invariant measure in the quotient $\mbox{GL}_{\gamma}(2, K_\fq)\backslash \mbox{GL}(2, K_\fq)$ that results from the choice of Haar measures in $\mbox{GL}(2, K_\fq)$ and $\mbox{GL}_{\gamma}(2, K_\fq)$ that give volume $1$ to their corresponding standard maximal open compact subgroup. Here $\mbox{GL}_{\gamma}(2, K_\fq)$ denotes the centralizer of $\gamma$. 

Embedding $K$ into its completion $K_\fq$ we can consider any matrix $\gamma\in\mbox{GL}(2, K)$ as an element of $\mbox{GL}(2, K_\fq)$. Doing this we can define for a regular elliptic matrix  
\begin{equation*}
    \cO(\gamma, f) := \displaystyle\prod_\fq \cO(\gamma, f_{\fq, k_\fq}).
\end{equation*}
where $\cO(\gamma, f_{\fq, k_\fq})$ is a local orbital integral as in \eqref{padicorbitalintegral} above.

Finally, in the same context as above, we also have the following 
\begin{definition}
Let $\gamma$ be a regular elliptic matrix in $\mbox{GL}(2, K)$. We say $\gamma$ is contributing to
\begin{equation*}
    f = \displaystyle\prod_\fq f_{\fq, k_\fq},
\end{equation*}
if
\begin{equation*}
    \cO(\gamma, f) \neq 0.
\end{equation*}
\end{definition}
Notice that trace and determinant of a contributing matrix are algebraic integers. 

Let $\gamma\in \mbox{GL}(2, K)$ be a regular elliptic element such that its characteristic polynomial, which must be irreducible, is of the form 
\begin{equation*}
    X^2 - \tau X + \det(\gamma)
\end{equation*}
for some $\tau, \det(\gamma)\in \cO_K$. Adjoining the roots of this polynomial to $K$ we get a field $K_{\gamma}$ which has degree two over $K$. Furthermore, if we pick any finite prime $\fq$ of $K$ and localize and complete we obtain the local field $K_\fq$ and an associated quadratic extension $K_{\gamma}\otimes_K K_\fq$. This extension is not necessarily a field extension, but rather just a quadratic extension of reduced $K_\fq$- algebras. As such there exists a generator $\Delta_\fq$ for this extension, that is, $K_{\gamma}\otimes_K K_\fq = K_\fq[\Delta_\fq]$. We will identify $\gamma$ with one of the solutions to it's characteristic polynomial in $K_{\gamma}\otimes_K K_\fq$. 

We know $\gamma$ generates the global monogenic order $O_{K}[\gamma]$ inside of $K_{\gamma}$ as well as the local monogenic ones $O_{K_\fq}[\gamma]$ inside of $K_{\gamma}\otimes_K K_\fq$. Consequently it has associated a zeta function of for each of these orders. 

In the local case, a result proven in \cite{ZYun} as lemma 4.5, written in our context, is the following

\begin{proposition}\label{prop: valuesorb}
In the context of the previous discussion, we have then
\begin{equation*}
    \tilde{J}_{O_{K_\fq}[\gamma]}(1) = \cO(f_{\fq, k_\fq}, \gamma),
\end{equation*}
where $\tilde{J}_{O_{K_\fq}(\gamma)}$ is the function whose existence is guaranteed by proposition \ref{polynomial} and $f_{\fq, k_\fq}$ is the local factor of the function $f$ for the choice $k_\fq = \valfq(\det(\gamma))$.
\end{proposition}
To exemplify the above result we prove
\begin{proposition}\label{valuesintegrals}
Let $\gamma$ be a contributing matrix for $f_{\fq, k}$ within the setup above. The $\cO(\gamma, f_{\fq, k})$ is given as follows:
\begin{enumerate}
    \item[(a)] If $\gamma$ is split then it is $q^m$,
    \item[(b)] If $\gamma$ is not split and the associated extension is unramified then it is
    \begin{equation*}
        q^m\dfrac{q + 1}{q - 1} - \dfrac{2}{q - 1}.
    \end{equation*}
    \item[(c)] If $\gamma$ is not split and the associated extension is ramified then it is
    \begin{equation*}
        \dfrac{q^{m + 1}}{q - 1} - \dfrac{1}{q - 1}.
    \end{equation*}
\end{enumerate}
In here $m$ is an integer that depends on $\gamma$ and the prime $p$.
\end{proposition}
\begin{proof}
Substitute $X = 1$, or equivalently $s = 0$, in the expressions of the polynomials in theorem \ref{solutionrecurrence} above. After simplifying, we recover the desired values for the orbital integrals.
\end{proof}
\begin{remark}
    This appears as lemma 1 in \cite{LanBE04} for $K = \mathbb{Q}$. The proof follows by counting lattices in the tree of $SL(2, \mathbb{Q}_q)$. 
\end{remark}
  
In order to apply all our proven results notice that if $\gamma = a + b\Delta_\fq$ with $b = p^nu$, where $u\in O_K^*$, then
\begin{equation*}
    O_{K_\fq}[\gamma] = O_{K_\fq}[p^n\Delta_p].
\end{equation*}
Furthermore, it is easy to prove that $\valfq(b)$ is independent of $\Delta_\fq$ and only depends on $\gamma$.

That is, in the language of the previous section, the monogenic order $O_{K_\fq}[\gamma]$ belongs to the sequence of orders constructed with respect to $\Delta_\fq$ (although it is independent of $\Delta_\fq$) and its depth in the sequence depends only on $\gamma$. Hence, we know exactly the polynomial nature it has, thanks to theorem \ref{solutionrecurrence}. Concretely, we have
\begin{align*}
    R_n(X) &= R_{n - 1}(X) + q^nX^{2n}\\
    U_n(X) &= (1 + X)R_{n - 1}(X) + q^nX^{2n},\\
    S_n(X) &= (1 - X)R_{n - 1}(X) + q^nX^{2n},
\end{align*}
where $R_n(X) = 1 + qX^2 + q^2X^4 + ... + q^nX^{2n}$. 
\begin{remark}
    For convenience, to make the above recurrence hold when $n = 0$, we define $R_{-1} = S_{-1} = U_{-1} = 0$. 
\end{remark}

Associated to the quadratic extension $K_{\gamma}/K$ there is a quadratic Hecke character defined as follows:
\begin{definition}
The quadratic sign character associated to the extension $K_{\gamma}/K$ is given at primes $\fq$ of $K$ by
\begin{equation*}
    \chi_{\gamma}(\fq) = \left\{
	\begin{array}{ll}
		1   & \mbox{if } \fq \mbox{ splits,}  \\
		-1  & \mbox{if } \fq \mbox{ is inert,}\\
		0   & \mbox{if } \fq \mbox{ ramifies.}
	\end{array}
	\right.
\end{equation*}
\end{definition}
This is of course the character associated to quadratic reciprocity. Using this character and the knowledge of the polynomials of the orders, we obtain immediately
\begin{equation}\label{eqn: uniformexpansion J}
     \tilde{J}_{O_{K_\fq}[\gamma]}(s) = q^{n_\fq s}\left(\left(1 - \dfrac{\chi_{\gamma}(\fq)}{q^{s}}\right)R_{n_\fq-1}(s) + q^{n_\fq}q^{-2n_\fq s}\right)
\end{equation}
where $n_\fq = \valfq(b)$ in the expansion $\gamma = a + b\Delta_\fq$. Notice that the polynomials are different in each case, yet it is the context in which we are (\textit{i.e.} all is constructed out of a regular elliptic matrix and its associated quadratic extension) which allow us via the Hecke Character to give a uniform description of the polynomial regardless of the type of local extension. Colloquially, we might say the Hecke character allow us to merge the polynomials into a single family.

It is this uniformity which allow us to proceed toward our main result. Concretely, notice that the divisors of $\fq^{n_\fq}$ are precisely $1, \fq, ..., \fq^{n_\fq}$. Hence we can write, using the explicit definition of $R_{n_\fq - 1}$, the following equation:
\begin{equation}\label{localJ}
   \tilde{J}_{O_{K_\fq}[\gamma]}(s) = N_K(\fq)^{n_\fq s}\displaystyle\sum_{\fd|\fq^{n_\fq}}N_K(\fd)^{(1 - 2s)}\displaystyle\prod_{\fp|\fd'}\left(1 - \dfrac{\chi_{\gamma}(\fp)}{N_K(\fp)^{s}}\right),
\end{equation}
where $\fd' := \fq^{n_\fq}/\fd$ and where $\fp$ runs over prime divisors of $\fd'$.
This motivates us to multiply over all primes to give the following
\begin{definition}\label{globals}
In the context of the previous discussion, define the function 
\begin{equation}
    \cO(s, \gamma, f) = \displaystyle\prod_{\fq} \tilde{J}_{O_{K_\fq}[\gamma]}(s),
\end{equation}
where $f$ is the function whose local factors are the $f_{\fq, k_\fq}$ as defined above.
\end{definition}

In order to invoke the multiplicativity structure that we have deduced for the functions $\tilde{J}_{O_{K_\fq}[\gamma]}(s)$ we need to know that the product in equation (\ref{globals}) is finite and to describe who are the primes that contribute. The main step to do this is the following fact of algebraic number theory, which we state for our case only (see \cite{LangANT}, proposition 13, chapter III):

\begin{proposition}\label{lang}
Let $E$ be an algebraic number field, $O_E$ its ring of integers, $F$ a finite extension of $E$ and $O_F$ the integral closure of $O_E$ in $F$. For $\alpha \in O_F$ and $\fq$ a prime of $E$, if $\fq$ does not divide 
\begin{equation*}
   \dfrac{\Disc_{F/E}(\alpha)}{\Disc(F/E)}
\end{equation*}
then $O_{F_\fq} = O_{E_\fq}[\alpha]$. In here $\Disc(F/E)$ is the relative discriminant of the extension $F/E$ while $\Disc_{F/E}(\alpha)$ is the discriminant of the element $\alpha$.
\end{proposition}
The general theory of free modules and discriminants in number fields implies that this quotient of discriminants is actually the square of some ideal $S$ (see in \cite{LangANT} the discussion at the start of chapter III).

Returning to our discussion with the regular elliptic element $\gamma$ and the quadratic extension $K_{\gamma}/K$ it generates we have the following
\begin{definition}
For the regular elliptic element $\gamma$ define
\begin{equation*}
    S_{\gamma}^2 = \dfrac{\Disc_{K_{\gamma}/K}(\gamma)}{\Disc(K_{\gamma}/K)}.
\end{equation*}
\end{definition}

\begin{proposition}\label{gammalands}
Let $S_{\gamma} = \displaystyle\prod_\fq \fq^{n_\fq}$ be the prime ideal factorization of $S_{\gamma}$, where the product runs over the primes of $O_K$. Then for each prime $\fq$, $n_\fq = \valfq(b)$ where $\gamma = a + b\Delta_\fq$. 
\end{proposition}
\begin{proof}
Localizing at the prime $\fq$ we get
\begin{equation*}
    \fq^{2n_\fq} = \dfrac{\Disc_{K_\fq(\gamma)/K_\fq}(\gamma)}{\Disc(K_\fq(\gamma)/K_\fq)},
\end{equation*}
where $K_\fq(\gamma):= K_{\gamma}\otimes_K K_\fq$. (We are slightly abusing of notation here since in the split case $K_p(\gamma)$ is not a quadratic extension).
In the local picture we have an $O_K$ integral basis for $K_\fq(\gamma)$, precisely $1, \Delta_\fq$, and so the discriminant can be computed with the standard formula with the Galois Conjugates (notice that in the split case this amounts to exchange of the coordinates). If we denote by $\overline{\Delta_\fq}$ the conjugate of $\Delta_\fq$ we get
\begin{equation*}
    \Disc(K_\fq(\gamma)/K_\fq) = (\Delta_\fq - \overline{\Delta_\fq})^2.
\end{equation*}
Analogously, for the discriminant of $\gamma$ we get
\begin{equation*}
    \Disc_{K_\fq(\gamma)/K_\fq}(\gamma) = (\gamma - \overline{\gamma})^2.
\end{equation*}
Finally, taking conjugate to $\gamma = a + b\Delta_\fq$ we get $\overline{\gamma} = a + b\overline{\Delta_\fq}$, which upon substraction leads to
\begin{equation*}
    (\gamma - \overline{\gamma}) = b(\Delta_\fq - \overline{\Delta_\fq}).
\end{equation*}
Squaring and taking valuation we conclude $\valfq(b) = n_\fq$, as desired.
\end{proof}
\begin{remark}
    This proposition shows that for a given $\gamma$, the integer $m$ in proposition \ref{valuesintegrals} is $n_\fq$.
\end{remark}

We get the following immediately:
\begin{proposition}
Let $\fq$ a be a prime of $K$. If $\fq$ doesn't divide $S_{\gamma}$ then 
\begin{equation*}
    \tilde{J}_{O_{K_\fq}[\gamma]}(s) = 1.
\end{equation*}
As a consequence, we have
\begin{equation*}
    \cO(s, \gamma, f) = \displaystyle\prod_{\fq|S_{\gamma}} \tilde{J}_{O_{K_\fq}[\gamma]}(s).
\end{equation*}
\end{proposition}
\begin{proof}
Putting $n_\fq = 0$ in equation \eqref{eqn: uniformexpansion J} we get the claim.
\end{proof}
The previous propositions imply two things: firstly, the product that defines $\cO(s, \gamma, f)$ is finite and hence the theory of arithmetic functions will let us \textit{glue} the product into a sum over the divisors of some ideal given by 
\begin{equation*}
    \displaystyle\prod_{\fq|S_{\gamma}}\fq^{n_\fq},
\end{equation*}
and \textit{a priori} this might not be $S_{\gamma}$. By itself this could be a minor point but in truth is very relevant. It is what will allow compatibility of all the different local parts and return an object that can be computed from the knowledge of $\gamma$ alone without having to find a prime factorization of $S_{\gamma}$.
Another instance of the importance of this regards the manipulations of the trace formula as has been discussed in the introduction. It is the fact that this ideal is $S_{\gamma}$ what allows the formula of Langlands \eqref{langlandsformula} to be compatible with the other terms appearing in the trace formula.

After this discussion we have

\begin{theorem}\label{generalformula}
Let $k_\fq$ be a nonnegative integer for each prime $\fq$ of $K$ with $k_\fq = 0$ for all but finitely many $\fq$. Let
\begin{equation*}
    f = \displaystyle\prod_\fq f_{\fq, k_\fq},
\end{equation*}
and $\gamma$ a contributing regular elliptic matrix for this $f$.  
We have the expansion
\begin{equation*}
    \cO(s, \gamma, f) := \displaystyle\prod_{\fq} \tilde{J}_{O_{K_\fq}[\gamma]}(s) = N_K(S_{\gamma})^s\displaystyle\sum_{\fd|S_{\gamma}}N_K(\fd)^{(1 - 2s)}\displaystyle\prod_{\fp|\fd'}\left(1 - \dfrac{\chi_{\gamma}(\fp)}{N_K(\fp)^{s}}\right), 
\end{equation*}
where $d' := \dfrac{S_{\gamma}}{d}$.  It is entire and satisfies the functional equation
\begin{equation*}
    \cO(1-s, \gamma, f) = \cO(s, \gamma, f).
\end{equation*}
In particular, evaluating at $s = 1$, we get
\begin{equation*}
    \displaystyle\prod_\fq \cO(\gamma, f_{\fq, k_\fq}) = \displaystyle\prod_{\fq|S_{\gamma}} \cO(\gamma, f_{\fq, k_\fq}) = \displaystyle\sum_{\fd|S_{\gamma}}N_K(\fd)\displaystyle\prod_{\fp|\fd}\left(1 - \dfrac{\chi_{\gamma}(\fp)}{N_K(\fp)}\right)
\end{equation*}
\end{theorem}
\begin{proof}
Using the standard theory of multiplicative arithmetic functions we get from equation (\ref{localJ}), now that we have a finite product, that
\begin{align*}
    \cO(s, \gamma, f) &= \displaystyle\prod_{\fq} \tilde{J}_{O_{K_\fq}[\gamma]}(s)\\
    &= \displaystyle\prod_{\fq|S_{\gamma}}\left(N_K(\fq)^{n_\fq s}\displaystyle\sum_{\fd|\fq^{n_\fq}}N_K(\fd)^{(1 - 2s)}\displaystyle\prod_{\fp|(\fq^{n_\fq}/\fd)}\left(1 - \dfrac{\chi_{\gamma}(\fp)}{N_K(\fp)^{s}}\right)\right)\\
    &= N_K\left(\displaystyle\prod_{\fq|S_{\gamma}}\fq^{n_\fq}\right)^s\left(\displaystyle\sum_{\fd|\prod_{\fq|S_{\gamma}}\fq^{n_\fq}}N_K(\fd)^{(1 - 2s)}\prod_{\fp|\fd'}\left(1 - \dfrac{\chi_{\gamma}(\fp)}{N_K(\fp)^{s}}\right)\right)\\
    &= N_K(S_{\gamma})^s\displaystyle\sum_{\fd|S_{\gamma}}N_K(\fd)^{(1 - 2s)}\displaystyle\prod_{\fp|\fd'}\left(1 - \dfrac{\chi_{\gamma}(\fp)}{N_K(\fp)^{s}}\right).
\end{align*}
It is entire and satisfies the functional equation since each one of  $\tilde{J}_{O_{K_\fq}[\gamma]}(s)$ does, by proposition (\ref{polynomial}).
Finally, evaluating at $s = 1$ and making the change of variable $\fd$ for $\fd'$ (both go over the divisors of $S_{\gamma}$), we obtain
\begin{equation*}
    \displaystyle\prod_\fq \cO(\gamma, f_\fq) =  \displaystyle\sum_{\fd|S_{\gamma}}N_K(\fd)\displaystyle\prod_{\fp|\fd}\left(1 - \dfrac{\chi_{\gamma}(\fp)}{N_K(\fp)}\right)
\end{equation*}
\end{proof}
With this we have recovered Langlands' formula \ref{langlandsformulaintro}, this time for a general number field. Furthermore, this settles Arthur's Conjecture \ref{arthurconjecture}, from the introduction, by using the previous theorem on $\mathbb{Q}$ and letting $k_p = 0$ for all primes $p$ of $\mathbb{Z}$ except for one arbitrary one. This settles the first of our two goals.


\section{The Congruence Conditions}\label{The Congruence Conditions}

Our objective in this section is to change the condition $d\mid S_{\gamma}$ for a different one which is more useful for applications. We begin with two lemmas. 

\begin{lemma}\label{lemma: exact exponent}
    Let $E/F$ be an extension of local fields. Let $\fp$ the prime of $F$, $\fP$ that of $E$ and let $e$ be the ramification exponent of this extension. If $E/F$ is tamely and totally ramified then the exponent with which $\fp$ divides $\mbox{Disc}(E/F)$ is $e-1$.
\end{lemma}
\begin{proof}
    This follows from standard theory of local fields. Concretely, if $\fD$ is the different of the extension then we know that the ideal norm satisfies
    \begin{equation*}
        N_{E/F}(\fD) = \Disc(E/F).
    \end{equation*}
    We also know that by definition
    \begin{equation*}
        N_{E/F}(\fP) = \fp^{f(E/F)},
    \end{equation*}
    where $f(E/F)$ is the inertia degree. Since $\fp$ totally ramified, the inertia degree is $1$. Thus 
    \begin{equation*}
        N_{E/F}(\fP) = \fp.
    \end{equation*}
    Hence, the exponent with which $\fp$ divides $\Disc(E/F)$ coincides with the exponent with which $\fP$ divides $\fD$. In the tamely ramified case this exponent is known to be $e - 1$. (For a proof of this see \cite{LangANT}, proposition 8, chapter III). This concludes the proof.
\end{proof}

\begin{lemma}\label{lemmanosolution}
Let $E/\mathbb{Q}_2$ be a finite field extension with ramification degree $e$. Let $\pi$ be a uniformizer of $E$ and $U$ any unit of $\cO_E$. Then, for any odd integer $1\le t \le 2e - 1$, there is no solution in $\cO_E$ to the congruence
\begin{equation*}
    X^2 - \pi^tU \equiv Y^2 \pmod{4}.
\end{equation*}
\end{lemma}
\begin{proof}
    Suppose such $X$ and $Y$ exist and rewrite the congruence as
    \begin{equation*}
        (X + Y)(X - Y) \equiv X^2 - Y^2 \equiv \pi^tU \pmod{4}.
    \end{equation*}
   If $\pi^e\mid X + Y$ or $\pi^e\mid X - Y$, then using that
   \begin{equation*}
       X + Y = X - Y + 2Y,
   \end{equation*}
   and that $\pi^e\mid 2$, we deduce $\pi^e \mid X \pm Y$. In such situation, we have
   \begin{equation*}
       \pi^{2e}|X^2 - Y^2,
   \end{equation*}
   which implies (since $4$ is $\pi^{2e}$ up to a unit) that
   \begin{equation*}
       \pi^{2e}\mid\pi^t,
   \end{equation*}
   but this is impossible for the range of values of $t$. Hence $\val_{\pi}(X\pm Y) < e$. 

   Now, using this last inequality, and the nonarchimidean triangle inequality for valuations we deduce
   \begin{align*}
       \val_{\pi}(X + Y)    &= \valpi(X - Y + 2Y)\\
                            &\ge \min\{\valpi(X - Y), e + \valpi(Y)\},
   \end{align*}
   but since the valuation of $X - Y$ and the $2Y$ are different (one being strictly smaller than $e$, and the other at least as big as $e$) we deduce there is equality. Hence
   \begin{equation*}
       \val_{\pi}(X + Y) = \valpi(X - Y)
   \end{equation*}
   Let us say this valuation is $0\le m < e$ and let us write
   \begin{equation*}
       X + Y = \pi^mV, X - Y = \pi^mW,
   \end{equation*}
   for some units $V, W$ in $\cO_E$.
   
   Then we have $\pi^{2m}VW$ and $\pi^tU$ have different valuations, because one is even and the other is odd. Also, both are smaller than $2e$. This implies that
    \begin{equation*}
        \valpi(\pi^{2m}VW - \pi^tU) = \min\{2m, t\} < 2e,
    \end{equation*}
    but this contradicts that 
   \begin{equation*}
       \pi^{2m}VW - \pi^tU \equiv X^2 - Y^2 - \pi^tU \equiv 0 \pmod{4}. 
   \end{equation*}
   This concludes the proof.
\end{proof}
 With this at hand we can now prove the following equivalence:
\begin{proposition}\label{prop: congruence conditions}
   Let $\gamma$ be a regular elliptic element in $GL(2, K)$. Then, for an ideal $I$ of $\cO_K$, the following facts are equivalent
   \begin{enumerate}
       \item $I$ divides $S_{\gamma}$.

       \item $I^2$ divides $\tau^2 - 4\det(\gamma)$ and for each prime $\fq$ of $\cO_K$ above $2$ we have
       \begin{equation*}
           \dfrac{\tau^2 - 4\det(\gamma)}{\pi_\fq^{2\valfq(I)}} \in \square_K(4).
       \end{equation*}
       Here $\square_K(4)$ is the set of squares classes modulo $4\cO_{K_\fq}$ and $\pi_\fq$ is a uniformizer of $\cO_{K_\fq}$.
   \end{enumerate}
\end{proposition}
\begin{proof}
    Recall the equality
    \begin{equation}\label{eqn: fundamental eqn}
        (\tau^2 - 4\det(\gamma)) = S_{\gamma}^2\Delta_{\gamma}.
    \end{equation}
    Let $\fq$ be any prime of $\cO_K$ and $\pi$ a uniformizer of $K_{\fq}$. In this proof we will write $\pi$ instead of $\pi_\fq$ to simplify notation. Taking valuation on both sides of the above equality, we deduce
    \begin{equation*}\label{eqn: fundamental eqn val}
        \valfq(\tau^2 - 4\det(\gamma)) = 2\valfq(S_{\gamma}) + \val\fq(\Delta_{\gamma}).
    \end{equation*}
    To simplify notation we will write $r:= \valfq(I)$.
    We prove both directions separately.
    
    \underline{(1) implies (2):}
    If $I$ divides $S_{\gamma}$ then the equation \eqref{eqn: fundamental eqn} above implies $I^2$ divides $\tau^2 - 4\det(\gamma)$.
    
    What remains to be proven is the extra fact about the primes above $2$. To that end, let $\fq$ be such a prime.
    If $\fq$ splits in the extension $K(\gamma)/K$ then in $K_{\fq}$ the characteristic polynomial of $\gamma$ splits. That is,
    \begin{equation*}
        \tau^2 - 4\det(\gamma)
    \end{equation*}
    is an integral square. Hence
    \begin{equation*}
        \dfrac{\tau^2 - 4\det(\gamma)}{\pi^{2\valfq(I)}},
    \end{equation*}
    is a square itself and thus also modulo $4$.
    
    If $\fq$ is not split, let $\Delta_{\fq}$ be a generator of the rings of integers of the extension $K_{\fq}(\gamma)/K_{\fq}$. Suppose further that it satisfies a quadratic equation
    \begin{equation*}
        X^2 - AX + B = 0,
    \end{equation*}
    for some $A, B\in\cO_{K_{\fq}}$. We thus know that for some unit $u$ we have
    \begin{equation*}
        \tau^2 - 4\det(\gamma) = \pi^{2n_\fq}u^2(A^2 - 4B),
    \end{equation*}
    where $n_\fq = \valfq(S_{\gamma})$. Using that $(1)$ holds we get $2r\le 2n_\fq$. Hence
       \begin{equation*}
           \dfrac{\tau^2 - 4\det(\gamma)}{\pi^{2r}} = \pi^{2(n_\fq-r)}u^2(A^2 - 4B) \equiv (\pi^{n_\fq-r}uA)^2 \pmod{4},
       \end{equation*}
       which is a square modulo $4$. This concludes the proof of this direction.
   
    \underline{(2) implies (1):} Let $I$ be an ideal that satisfies $(2)$. We must prove that for each prime ideal $\fq$ we have
    \begin{equation}\label{eqn: inequality to prove}
        r:=\valfq(I)\le\valfq(S_{\gamma}).
    \end{equation}
    
    We have three possibilities for $\fq$:
        \begin{description}
            \item[(a)] It is unramified.
            \item[(b)] It is ramified but not above $2$.
            \item[(c)] It is ramified and above $2$.
        \end{description}
    For case (a), we have $\valfq(\Delta_{\gamma}) = 0$. Using that (2) holds and equation \ref{eqn: fundamental eqn val} above we get
    \begin{equation*}
        2r \le \valfq(\tau^2 - 4\det(\gamma)) = 2\valfq(S_{\gamma}).
    \end{equation*}
    From this, $r\le \valfq(S_{\gamma})$ as desired.
   
    For case (b), we have that $\fq$ is tamely and totally ramified. Thus the inertia degree of $K_\fq(\gamma)/K_\fq$ is $1$. Using lemma \ref{lemma: exact exponent} we get
    \begin{equation*}
        \valfq(\Delta_{\gamma}) = 1.
    \end{equation*}
    Thus, from equation \eqref{eqn: fundamental eqn val}, we have
    \begin{equation*}
        2r \le \valfq(\tau^2 - 4\det(\gamma)) = 2\valfq(S_{\gamma}) + 1.
    \end{equation*}
    Since all the terms in the above equation are integers we have
    \begin{equation*}
        2r \le 2\valfq(S_{\gamma}),
    \end{equation*}
    and so $r\le \valfq(S_{\gamma})$. This concludes this case.
    
    For case (c) we proceed as follows: once more $\fq$ is totally ramified. As such, we have that there exists a uniformizer  of $\cO_{K_{\fq}(\gamma)}$ such that
    \begin{enumerate}
        \item it generates the ring of integers $\cO_{K_{\fq}(\gamma)}$,
        \item it satisfies an Eisenstein polynomial, which we denote once more as
        \begin{equation*}
        X^2 - AX + B = 0.
    \end{equation*}
        We write $A = \pi^tA_0$ and $B = \pi B_0$, with $A_0, B_0$ units and $t\ge 1$. We can do this by definition of Eisenstein polynomial. 
        \end{enumerate}
(For a proof of this fact, as well as of a discussion of Eisenstein polynomials, see \cite[section 5 chapter II]{LangANT}).

Using this uniformizer and its Eisenstein polynomial we get
    \begin{equation*}
        \tau^2 - 4\det(\gamma) = \pi^{2n_\fq}u^2(A^2 - 4B) = \pi^{2n_\fq}u^2(A_0^2\pi^{2t} - v^2\pi^{2e + 1} B_0),
    \end{equation*}
In the above equation we used $2 = \pi^ev$ for some unit $v$ and where $e$ is the ramification exponent of $\fq$ above $2$. Notice that
\begin{equation*}
    2r \le \valfq(\tau^2 - 4\det(\gamma)) = 2n_\fq + \valfq(A_0^2\pi^{2t} - v^2\pi^{2e + 1} B_0) = 2n_\fq + \min\{2t, 2e+1\},
\end{equation*}
where the last equality follows from the ultrametirc inequality, which is an equality in this case because the valuations of the terms are different.

Let us suppose for a contradiction that $m:= r - n_\fq \ge 1$. Hence
\begin{equation*}
    2m \le \min\{2t, 2e+1\}.
\end{equation*}
In particular $2m$ is not bigger than either of $2t$ or $2e + 1$. We thus have
\begin{equation*}
    \dfrac{\tau^2 - 4\det(\gamma)}{\pi^{2r}} = u^2(A_0^2\pi^{2t - 2m} - \pi^{2e+1-2m}vB_0^2)
\end{equation*}
Given that (2) holds, and that $\fq$ is a prime above $2$, we know this quotient is a square modulo $4$. That is
\begin{equation*}
    u^2(A_0^2\pi^{2t - 2m} - \pi^{2e+1-2m}vB_0^2) \in\square_K(4),
\end{equation*}
and since $u$ is invertible, then
 \begin{equation}\label{Is a square}
    A_0^2\pi^{2t - 2m} - \pi^{2e+1-2m}vB_0^2 \in\square_K(4).
\end{equation}   
Notice that $2e + 1 - 2m$ is an odd integer between $1$ and $2e-1$ and that $vB_0^2$ is a unit. Hence, lemma \ref{lemmanosolution} implies the above inclusion \eqref{Is a square} cannot occur. This contradiction shows $r\le n_\fq$. This concludes the proof of the theorem.
\end{proof}

\section{Generalization of Zagier's Zeta Function}\label{Generalization of Zagier's Zeta Function}
We begin with two lemmas that will aid in our construction. The first one is as follows:

\begin{lemma}\label{matrixofdelta}
Let $K$ be a number field and  $\delta\in\square_K(4)$ be a nonsquare algebraic integer of $\cO_K$. Then
\begin{enumerate}
    \item[(i)] There exists a regular elliptic matrix $\gamma\in GL(2, K)$ with entries in $\cO_K$ and $\tau^2 - 4\det(\gamma) = \delta$.
    \item[(ii)] The matrix $\gamma$ in (i) is a contributing matrix for 
    \begin{equation*}
    f = \displaystyle\prod_{\fq }f_{\fq,\; k_\fq}, 
\end{equation*}
    with $k_\fq = \valfq(\det(\gamma))$ for each prime $\fq$ of $\cO_K$.
\end{enumerate}
\end{lemma}
\begin{proof} 
    We begin with the proof of $(i).$ Because $\delta\in\square_K(4)$ there exist integers $m$ and $r$ in $\cO_K$ such that
    \begin{equation*}
        \delta = 4m + r^2.
    \end{equation*}
    We have $m\neq 0$ because $\delta$ is itself not a square. Consider the matrix
    \begin{equation*}
        \gamma = \begin{pmatrix}
            r & 1\\
        m & 0\\
        \end{pmatrix}.
    \end{equation*}
    The discriminant of the characteristic polynomial of this matrix is
    \begin{equation*}
        \mbox{Trace}(\gamma)^2 - 4\det(\gamma) = r^2 - 4(-m) = \delta.
    \end{equation*}
    We conclude it is regular elliptic since $\delta$ is not a square, so the characteristic polynomial is indeed irreducible.
    (ii) is immediate by the definition.
\end{proof}
We also have:
\begin{lemma}\label{lemma: well defined char}
Let $K$ be a number field and  $\delta\in\square_K(4)$ be a nonsquare algebraic integer of $\cO_K$. Let $\gamma_1, \gamma_2$ be two regular elliptic matrices such that
\begin{equation*}
    \mbox{Trace}(\gamma_1)^2 - 4\det(\gamma_1) =\mbox{Trace}(\gamma_2)^2 - 4\det(\gamma_2) = \delta.
\end{equation*}
Also, let $f_1, f_2$ be functions to which $\gamma_1, \gamma_2$ contribute respectively. Then we have the following equalities
\begin{enumerate}
    \item $\chi_{\gamma_1} = \chi_{\gamma_2}$,

    \item $S_{\gamma_1} = S_{\gamma_2}$,

    \item $\cO(s, \gamma_1, f_1) = \cO(s, \gamma_2, f_2)$.
    \item $\cO_{K}[\gamma_1] = \cO_K[\gamma_2]$.
\end{enumerate}
\end{lemma}
\begin{proof}
    The quadratic extension generated by $\gamma_1$ and $\gamma_2$ coincide. Concretely, it is $K(\sqrt{\delta})/K$. The character depends on the extension itself and not in the polynomial that generates it. This proves (i).

    Using the definition of $S_{\gamma_1}$ and $S_{\gamma_2}$, we have
    \begin{equation*}
        (\mbox{Trace}(\gamma_1)^2 - 4\det(\gamma_1)) = S_{\gamma_1}^2\Delta_{\gamma_1},
    \end{equation*}
    and 
    \begin{equation*}
        (\mbox{Trace}(\gamma_2)^2 - 4\det(\gamma_2)) = S_{\gamma_2}^2\Delta_{\gamma_2}.
    \end{equation*}
    We know that $\Delta_{\gamma_1} = \Delta_{\gamma_2}$ because it is the relative discriminant of the extension $K(\sqrt{\delta})/K$ which depends only on $\delta$. Thus, from the equations above, we deduce
    \begin{equation*}
        S_{\gamma_1}^2 = \dfrac{(\delta)}{\Delta_{\gamma_1}} = \dfrac{(\delta)}{\Delta_{\gamma_2}}= S_{\gamma_2}^2.
    \end{equation*}
    From this (ii) follows. 

    Using theorem \ref{generalformula}, and parts $(1)$ and $(2)$ of this lemma, we get
    \begin{align*}
    \cO(s, \gamma_1, f_1) 
    &= N_K(S_{\gamma_1})^s\displaystyle\sum_{I|S_{\gamma_1}}N_K(I)^{(1 - 2s)}\displaystyle\prod_{\fp|I'}\left(1 - \dfrac{\chi_{\gamma_1}(\fp)}{N_K(\fp)^{s}}\right)\\
    &= N_K(S_{\gamma_2})^s\displaystyle\sum_{I|S_{\gamma_2}}N_K(I)^{(1 - 2s)}\displaystyle\prod_{\fp|I'}\left(1 - \dfrac{\chi_{\gamma_2}(\fp)}{N_K(\fp)^{s}}\right)\\
    &= \cO(s, \gamma_2, f_2).
\end{align*}
Finally, using proposition \ref{gammalands} and $(2)$, we know $\cO_{K}[\gamma_1]$ and $\cO_{K}[\gamma_2]$ localize to the same local orders at every prime. Hence,
\begin{equation*}
    \cO_{K}[\gamma_1] = \cO_K[\gamma_2].
\end{equation*}
\end{proof}
\begin{remark}
    One might be tempted to guess that if $\gamma_1$ and $\gamma_2$ are as above, then one can construct $f_1 = f_2$. This not true because the functions are built out of the determinants which are different. 
    
    For example, consider $K = \mathbb{Q}$ and $\delta = 45$. The matrices
    \begin{equation*}
        \gamma_1 = \begin{pmatrix}
            3 & 1\\
            9 & 0\\
        \end{pmatrix}\;,
        \gamma_2 = \begin{pmatrix}
            5 & 1\\
            5 & 0\\
        \end{pmatrix}.
    \end{equation*}
    have characteristic polynomials $X^2 - 3X - 9$ and $X^2 - 5X - 5$, respectively. Both generate the extension $\mathbb{Q}(\sqrt{5})/\mathbb{Q}$. Nevertheless, $f_1\neq f_2$. They differ at $p = 3$ and $p = 5$. Indeed $f_1$ has local factor at $5$ equal to $f_{5, 0}$ while $f_2$ has local factor at $5$ equal to $f_{5, 1}$. On the other hand, at $p = 3$ the local factor of $f_1$ is $f_{3, 2}$ while that of $f_2$ is $f_{3, 0}$. They coincide at every other place. 

    For the prime $p = 3$, the extension $\mathbb{Q}_3(\sqrt{5})/\mathbb{Q}_3$ is unramified and $\sqrt{5}$ generates te ring of integers. Expanding $\gamma_1, \gamma_2$ with respect to the corresponding integral basis we have
    \begin{align*}
        \gamma_1 &= -\dfrac{3}{2} + \dfrac{3}{2}\sqrt{5},\\
        \gamma_2 &= -\dfrac{5}{2} + \dfrac{3}{2}\sqrt{5}.
    \end{align*}
    Since $\val_{3}(3/2) = 1$, we deduce their contribution at the prime $3$ is the same. Furthermore, as we know from proposition \ref{gammalands}, we must have 
    \begin{equation*}
        \val_3(S_{\gamma_1}) = \val_3(S_{\gamma_2}) = 1.
    \end{equation*}
     Indeed, the relative discriminant of the extension $\mathbb{Q}(\sqrt{5})/\mathbb{Q}$ is $(5)$, and so
    \begin{equation*}
        S_{\gamma_i}^2 = \dfrac{(45)}{(5)} = (3)^2, \; i = 1, 2.
    \end{equation*}
    Hence $S_{\gamma_i} = (3)$. A similar analysis can be done at the other primes.
\end{remark}

We now give the following:
\begin{definition}
 Let $K$ be a number field and  $\delta\in\square_K(4)$ be a nonsquare algebraic integer of $\cO_K$. The common quadratic character, constructed out of regular elliptic matrices $\gamma$ as above, will be denoted by $\chi_{\delta}$. Also, the common $S_{\gamma}$ associated to the common quadratic extensions $K_{\gamma}/K$ will be denoted by $S_{\delta}$. The common $\cO(s, \gamma, f)$ built as above will be denoted by $\cO(s, \delta)$. Finally, the common order will be denoted by $\cO_{\delta}$.
\end{definition}
\begin{remark}
    Notice that lemma \ref{matrixofdelta} guarantees the definitions above are not vacuous. On the other hand, lemma \ref{lemma: well defined char} implies they are well defined.
\end{remark}

In order to proceed we need to introduce some notation. We begin by a definition which is motivated by proposition \ref{prop: congruence conditions} above.

\begin{definition}
 Let $K$ be a number field and $\delta\in\square_K(4)$ be a nonsquare algebraic integer of $\cO_K$. Let $I$ be an ideal of $\cO_K$. We say $I$ \textbf{satisfies the congruence conditions for $\delta$} if
 \begin{enumerate}
     \item $I^2$ divides $\delta$,
     \item for each prime $\fq$ of $\cO_K$ above $2$ we have
       \begin{equation*}
           \dfrac{\delta}{\pi_{\fq}^{2\valfq(I)}} \in \square_K(4),
       \end{equation*} 
       where $\pi_\fq$ is a uniformizer of $K_\fq$.  
 \end{enumerate}
 When there is no danger of confusion we shall simply say \textbf{I satisfies the congruence conditions.}
\end{definition}
 We can restate proposition \ref{prop: congruence conditions} in this context as follows:
 \begin{proposition}
   Let $K$ be a number field and $\delta\in\square_K(4)$ be a nonsquare algebraic integer of $\cO_K$.  Let $I$ be an ideal of $\cO_K$. Then the following are equivalent:
   \begin{enumerate}
        \item $I|S_{\delta}$
       \item $I$ satisfies the congruence conditions for $\delta$.
   \end{enumerate}
 \end{proposition}
\begin{proof}
    This is immediate from proposition \ref{prop: congruence conditions} and the definition of $S_{\delta}$.
\end{proof}

Let $\delta\in\square_K(4)$ be a nonsquare algebraic integer of $\cO_K$.  The quadratic character $\chi_{\delta}$ is a primitive character of conductor $\Delta_{\gamma}$, for any regular elliptic matrix $\gamma$ used to generate the extension. Associated to an ideal $I$ that satisfies the congruence conditions we can define a possibly nonprimitive character as follows:

\begin{definition}
   Let $K$ be a number field and $\delta\in\square_K(4)$ be a nonsquare algebraic integer of $\cO_K$. Let $I$ be an ideal that satisfies the congruence conditions for $\delta$. Define at a prime ideal $\fq$ of $\cO_K$,
    \begin{equation*}
    \chi_{I}(\fq) := \left\{
	\begin{array}{ll}
		\chi_{\delta}(\fq)   & \mbox{if } \fq \mbox{ does not divide } S_{\delta}/I \\
		0  & \mbox{if } \fq \mbox{ divides } S_{\delta}/I
	\end{array}
	\right.
\end{equation*}
We extend this definition to all ideals of $\cO_K$ by complete multiplicativity. 
\end{definition}

Using this we define the generalization of Zagier's zeta function. 
\begin{definition}\label{defn: zagier zeta}
   Let $K$ be a number field and  $\delta\in\square_K(4)$ be a nonsquare algebraic integer of $\cO_K$. We define Zagier's zeta function as 
   \begin{equation}
    L(s, \delta) := \displaystyle\sum_{I^2|\delta}'\dfrac{1}{N_K(I)^{2s - 1}}L\left(s, \chi_{I}\right).
\end{equation}
 The sum is taken only over those $I$ such that $I^2$ divides $\delta$ that satisfy the congruence conditions. This is specified in the sum by the $'$ on top.   
\end{definition}
\begin{remark}
    We are following \cite{AliI} in his use of $'$ to denote the congruence conditions. The notation is sligthly redundant since $I^2|\delta$ is part of the congruence conditions. We keep it like this to agree with how this notation was used before in the literature for the case of $\mathbb{Q}$.
\end{remark}
We have the following result relating to the properties of this function.

\begin{theorem}\label{thm: zagier zeta properties}
The function $L(s, \delta)$ defined above satisfies the following properties
\begin{enumerate}
    \item It can be analytically continued to an entire function $\Lambda(s, \delta)$ by defining
    \begin{equation*}
        \Lambda(s, \delta) := N_K(S_{\delta})^sL_{\infty}(s, \chi_{\delta})L(s, \delta).
    \end{equation*}
    Here $L_{\infty}(s, \chi_{\delta})$ is the product of the $L-$factors at the infinite places of $K$ for the character $\chi_{\delta}$.
    \item  We have
    \begin{equation*}\label{relationzetaseq}
        \Lambda(s, \delta ) = \Lambda(s, \chi_{\delta})\cO(s, \delta),
    \end{equation*}
   where $\Lambda(s, \chi_{\delta})$ is the standard completion of the $L-$function of the character $\chi_{\delta}$ to an entire function.

    \item We have the functional equation
    \begin{equation*}
        \Lambda(s, \delta) = \Lambda(1 - s, \delta).
    \end{equation*}

    \item We have the following relationship between the different global zeta functions
    \begin{equation*}
        \Lambda(s, \delta) = \Lambda(s, \chi_{\delta})\cdot\dfrac{\Lambda_{\cO_{\delta}}(s)}{\Lambda_{K(\sqrt{\delta})}(s)}.
    \end{equation*}
\end{enumerate}
\end{theorem}
\begin{proof}
   For a given ideal $I$ satisfying the congruence conditions, we define $I' = S_{\gamma}/I$. We can relate the $L$-functions of $\chi_I$ and $\chi_{\delta}$ by the equation
   \begin{equation*}
       L(s, \chi_I) = L(s, \chi_{\delta})\displaystyle\prod_{\fq|I'}\left(1 - \dfrac{\chi_{\delta}(\fq)}{N_K(\fq)^s}\right).
   \end{equation*}
   We have
    \begin{align*}
        L(s, \delta) 
        &= \displaystyle\sum_{I^2|\delta}'\dfrac{1}{N_K(I)^{2s - 1}}L\left(s, \chi_{I}\right)\\
        &= \displaystyle\sum_{I^2|\delta}'\dfrac{1}{N_K(I)^{2s - 1}}\displaystyle\prod_{\fq|I'}\left(1 - \dfrac{\chi_{\delta}(\fq)}{N_K(\fq)^s}\right)L(s, \chi_{\delta})\\
        &= L(s, \chi_{\delta})\displaystyle\sum_{I^2|\delta}'\dfrac{1}{N_K(I)^{2s - 1}}\displaystyle\prod_{\fq|I'}\left(1 - \dfrac{\chi_{\delta}(\fq)}{N_K(\fq)^s}\right)\\
        &= L(s, \chi_{\delta})\displaystyle\sum_{I|S_{\delta}}\dfrac{1}{N_K(I)^{2s - 1}}\displaystyle\prod_{\fq|I'}\left(1 - \dfrac{\chi_{\delta}(\fq)}{N_K(\fq)^s}\right)\\
        &= N_K(S_{\delta})^{-s}L(s, \chi_{\delta})N_K(S_{\delta})^s\displaystyle\sum_{I|S_{\delta}}\dfrac{1}{N_K(I)^{2s - 1}}\displaystyle\prod_{\fq|I'}\left(1 - \dfrac{\chi_{\delta}(\fq)}{N_K(\fq)^s}\right)\\
        &= N_K(S_{\delta})^{-s}L(s, \chi_{\delta})\cO(s, \delta).
    \end{align*}
Notice that we have used that $I$ satisfies the congruence conditions for $\delta$ to change the index of summation. Multiplying by $L_{\infty}(s, \chi_{\delta})$ we get
\begin{equation*}
    \Lambda(s, \delta) = \Lambda(s, \chi_{\delta})\cO(s, \delta).
\end{equation*}
From this equation we get the four claims. (2) is the statement of the above equation. Furthermore, both factors on the right are entire and satisfy the functional equation, one by standard theory of characters and the other by theorem \ref{generalformula} above. 

Finally, (4) follows immediately from (2) and proposition \ref{globalzeta} above. 
\end{proof}
\begin{remark}
    We have spent a considerable amount of effort in this section to prove that the choices of $\gamma$, and thus of $f$, do not change the definitions associated to a given nonsquare $\delta\in\square_K(4)$. This might misled one to think that $f$ is not really important as it is specified by the given $\gamma$, but this is imprecise.

    In the applications of this section, to construct the generalization of Zagier zeta function, we have started out from some appropriate $\delta$, then constructed $\gamma$ and then $f$. With all these, we can invoke the properties of $\cO(s, \gamma, f)$ and prove the choices were not relevant.  
    
    Nevertheless, in other applications, as those we shall see in the next section, the process goes in the opposite direction: we are given a test function, which then sifts those $\gamma$ that contribute to it. Usually the test function is constructed to satisfy properties desirable to the appropriate problem at hand. 
    
    In summary, neither parameter is superfluous, they are just in different ways according to the purposes at hand.
\end{remark}
With this at hand we can now discuss how the above construction recovers the original Zagier zeta function when $K = \mathbb{Q}$. Zagier's original function is
\begin{equation}\label{eqn: zagier original}
    L(s, \delta) := \displaystyle\sum_{f^2|\delta}'\dfrac{1}{f^{2s - 1}}L\left(s, \left(\dfrac{\frac{\delta}{f^2}}{\cdot}\right)\right).
\end{equation}
The sum is taken only over those $f$ such that $\delta/f^2 \equiv 0, 1\pmod 4$. Notice that $\square_{\mathbb{Q}}(4) = \{0, 1\}.$ In here
\begin{equation*}
    \left(\dfrac{\cdot}{\cdot}\right),
\end{equation*}
is the Kronecker Symbol. This symbol allows us to write explicitly the quadratic sign character of quadratic extensions over $\mathbb{Q}$. Concretely, we have
\begin{equation*}
    \chi_{\gamma}(\cdot) = \left(\dfrac{D_{\gamma}}{\cdot}\right),
\end{equation*}
where $D_{\gamma}$ is the fundamental discriminant of the extension generated by $\gamma$ (i.e. the specific generator of $\Delta_{\gamma}$ obtained by computing with an integral basis). Hence, for each $f|S_{\gamma}$ we have
\begin{equation*}
    \chi_{f}(\cdot) = \left(\dfrac{\frac{\delta}{f^2}}{\cdot}\right),
\end{equation*}
so that equation \eqref{eqn: zagier original} is indeed the case for $K = \mathbb{Q}$ of definition \ref{defn: zagier zeta}. 
Furthermore, Zagier's zeta function has a completion, as follows:
\begin{equation}\label{completionzagierzeta}
    \Lambda_O(s, \delta): = \left(\dfrac{|\delta|}{\pi}\right)^{s/2}\Gamma\left(\dfrac{z + i_{\delta}}{2}\right)L(s, \delta),
\end{equation}
where $i_{\delta}$ is $0$ or $1$ according to whether $\delta$ is positive or negative, respectively. We have used the subindex $O$ to denote \textit{original} as opposed to the completion we are defining above via theorem \ref{thm: zagier zeta properties}. 

Notice that $\Lambda_O$ depends only on the sign $\delta$ itself without using the character $\chi_{\delta}$. Instead, in theorem \ref{thm: zagier zeta properties} we used the properties of the character to define the completion. What we must verify is that both of these approaches coincide (as opposed of one being an entire multiple of the other). Indeed we have
\begin{proposition}
   In the context of the discussion above, both completions of Zagier's zeta function coincide. That is,
   \begin{equation*}
       \Lambda_O(s, \delta) = \Lambda(s, \delta).
   \end{equation*}
\end{proposition}
\begin{proof}
    The computations above, and the Euler Factor expansion of the $L$-function of a character, imply for $K = \mathbb{Q}$
\begin{equation}\label{equalitylfunctions}
  L(s, \chi_{\gamma})O(s, \gamma, f) = S_{\gamma}^s\displaystyle\sum_{d^2\mid \;\delta}^{'}\dfrac{1}{d^{2s - 1}}L\left(s, \left(\dfrac{\frac{\delta}{d^2}}{\cdot}\right)\right) = S_{\gamma}^sL\left(s,\delta \right). 
\end{equation}
Notice that the $L$ function on the left hand side is one associated with a quadratic sign character $\chi_{\gamma}$ while that on the right is associated to the discriminant $\tau^2 - 4\det(\gamma) = \delta$.

In order to get the equality let us complete the right hand side following equation \eqref{completionzagierzeta} above. We have
\begin{align*}
   \Lambda_O\left(s,\delta \right)
   &= \left(\dfrac{|\delta|}{\pi}\right)^{s/2}\Gamma\left(\dfrac{s + i_{\delta}}{2}\right)L(s, \delta)\\
   &= \dfrac{|S_{\gamma}^{2}D_{\gamma}|^{s/2}}{\pi^{s/2}}\Gamma\left(\dfrac{s + i_{\delta}}{2}\right)L(s, \delta)\\
   &= S_{\gamma}^{s} |D_{\gamma}|^{s/2}\pi^{-s/2}\Gamma\left(\dfrac{s + i_{\delta}}{2}\right)L(s, \delta)\\
   &= S_{\gamma}^{s} |D_{\gamma}|^{s/2}\pi^{-s/2}\Gamma\left(\dfrac{s + i_{\delta}}{2}\right)S_{\gamma}^{-s}L(s, \chi_{\gamma})O(s, \gamma, f)\\
   &= |D_{\gamma}|^{s/2}\pi^{-s/2}\Gamma\left(\dfrac{s + i_{\delta}}{2}\right)L(s, \chi_{\gamma})O(s, \gamma, f).
\end{align*}
Recall that the completion of the $L-$function associated to $\chi_{\gamma}$ is
\begin{equation*}
    \left(\dfrac{|D_{\gamma}|}{\pi}\right)^{s/2}\Gamma\left(\dfrac{s + i_{\gamma}}{2}\right)L(s, \chi_{\gamma}),
\end{equation*}
with $i_{\gamma} = 0$ or $1$ according to $\chi_{\gamma}(-1) = 1$ or $\chi_{\gamma}(-1) = -1$. By definition the Kronecker Symbol satisfies for an integer $n$
\begin{equation*}
    \left(\frac{n}{-1}\right) =\left\{ \begin{array}{cc}
         1&   n \ge 0,\\
         -1&  n < 0.
    \end{array}\right.
\end{equation*}
Since $\delta = S_{\gamma}^2D_{\gamma}$, it follows $\delta$ and $D_{\gamma}$ have the same sign. Consequently the conditions for the value of $i_{\delta}$ and $i_{\gamma}$ coincide. We conclude
\begin{equation*}
    \Gamma\left(\dfrac{s + i_{\delta}}{2}\right) = \Gamma\left(\dfrac{s + i_{\gamma}}{2}\right).
\end{equation*}
Using this in the computations above, we conclude
\begin{equation*}
    \Lambda_O\left(s,\delta \right) =  \Lambda(s, \chi_{\gamma}) O(s, \gamma, f) = \Lambda(s, \gamma),
\end{equation*}
as desired.
\end{proof}
An immediate consequence of the above is the recovery of the well known properties of the original Zagier's zeta function.
\begin{corollary}\label{functionaleqn}
    Let $\delta \equiv 0, 1\pmod 4$ be a nonsquare integer, then  $\Lambda( s, \delta)$ is entire and satsfies the functional equation
    \begin{equation*}
        \Lambda(s, \delta) = \Lambda(1 - s, \delta).
    \end{equation*}
\end{corollary}
\begin{proof}
This is immediate from theorem \ref{thm: zagier zeta properties} and the equality $\Lambda_O(s, \delta) = \Lambda(s, \delta)$.
\end{proof}


\section{Relevance of our work in the strategy of BE}\label{Relevance of our work in the strategy of BE}

The results we prove are very relevant for the generalization to arbitrary number fields of the methods in \cite{AliI}. We will discuss their use only over $\mathbb{Q}$, to simplify the exposition. The use in arbitrary number fields is the same and the role our results play is the strategy is in proving the required results hold. We refer to the interested reader to \cite{BEgenfields} for a deeper discussion over other number fields.

We will fix the prime $p$ in this subsection and define
\begin{equation*}
    f = \displaystyle\prod_{q < \infty}f_q \times f_{\infty}, 
\end{equation*}
where at finite primes $q\neq p$ we  have $f_q$ is the indicator function of the maximal open compact in $\mbox{GL}(2, \mathbb{Q}_q)$ and
\begin{equation*}
    f_p = p^{-k/2}\mathbf{1}\left(X\in\mbox{Mat}(2, \mathbb{Z}_p) \mid |\det(X)|_p = p^{-k}\right),
\end{equation*}
where $k$ is a nonnegative integer. In the notation used above, $f_p = p^{-k/2}f_{p, \;k}$ while for $p\neq q$, $f_q = f_{q, 0}$. Using theorem \ref{generalformulaintro} we get
\begin{equation}\label{langlandsformula}
    \displaystyle\prod_{q<\infty} \cO(\gamma, f_q) = p^{-k/2}\displaystyle\sum_{d|S_{\gamma}} d\displaystyle\prod_{q|d}\left(1 -  \dfrac{\left(\frac{D_{\gamma}}{q}\right)}{q}\right).
\end{equation}

Formula \eqref{langlandsformula} is useful because it allows the finite part of the orbital integral to merge with the volume term. Concretely, it can be proven that
\begin{equation*}
    \mbox{vol}(\gamma) = \sqrt{|D_{\gamma}|} L(1, \chi_{\gamma}),
\end{equation*}
where $\chi_{\gamma}(\cdot) = \left(\frac{D_{\gamma}}{\cdot}\right)$. Hence, each of the summands  
\begin{equation*}
    \mbox{vol}(\gamma)\cO(\gamma, f).
\end{equation*}
of the regular elliptic part becomes
\begin{equation*}
    p^{-k/2}\sqrt{|D_{\gamma}|} L(1, \chi_{\gamma})\left(\displaystyle\sum_{d|S_{\gamma}} d\displaystyle\prod_{q|d}\left(1 -  \dfrac{\left(\frac{D_{\gamma}}{q}\right)}{q}\right)\right) \cO(\gamma, f_{\infty})
\end{equation*}
Upon manipulation with the definition of $S_{\gamma}$ this gets transformed into
\begin{equation*}
    S_{\gamma}^{-1}L(1, \chi_{\gamma})\left(\displaystyle\sum_{d|S_{\gamma}} d\displaystyle\prod_{q|d}\left(1 -  \dfrac{\left(\frac{D_{\gamma}}{q}\right)}{q}\right)\right)\cdot p^{-k/2}|\tau^2\mp 4p^k|^{1/2}\cO(\gamma, f_{\infty}).
\end{equation*}
The Poisson Summation we desire to apply should be performed over the trace $\tau$. In our present case, $\tau$ is an integer so we could proceed if we knew two things. First, that we actually had a complete lattice as required by the Poisson Summation formula. Secondly, that we are dealing with a sum over this complete lattice of a Schwartz function evaluated at the points of the lattice. The whole point of \cite{AliI} is to overcome these two issues and then perform Poisson Summation to find the contribution of the trivial representation.

Altu\u{g} proves is that after several nontrivial manipulations, the term
\begin{equation*}
    p^{-k/2}|\tau^2\mp 4p^k|^{1/2}\cO(\gamma, f_{\infty})
\end{equation*}
is indeed a function evaluated at the integer $\tau$ as required. The issue is that it is not a Schwartz function as it has the singularities introduced by the orbital integral. Furthermore,
\begin{equation*}
    S_{\gamma}^{-1}L(1, \chi_{\gamma})\left(\displaystyle\sum_{d|S_{\gamma}} d\displaystyle\prod_{q|d}\left(1 -  \dfrac{\left(\frac{D_{\gamma}}{q}\right)}{q}\right)\right)
\end{equation*}
is just a number, not an actual Schwartz function evaluated at some integer. Altu\u{g} overcomes these two issues by invoking Zagier's zeta function. What Altu\u{g} realizes is
\begin{equation*}
    L(1, \tau^2\mp 4p^k) = S_{\gamma}^{-1}L(1, \chi_{\gamma})\left(\displaystyle\sum_{d|S_{\gamma}} d\displaystyle\prod_{q|d}\left(1 -  \dfrac{\left(\frac{D_{\gamma}}{q}\right)}{q}\right)\right) 
    = \displaystyle\sum_{d^2|\;\delta}'\dfrac{1}{d}L\left(1, \left(\dfrac{\frac{\delta}{d^2}}{\cdot}\right)\right).
\end{equation*}
This is how Zagier's zeta function appears in the strategy. We proved the corresponding equation for general number fields in section \ref{Generalization of Zagier's Zeta Function} above. 
The regular elliptic part becomes 
\begin{equation}\label{alimanipulation}
    \displaystyle\sum_{\pm\tau} p^{-k/2}L(1, \tau^2\mp 4p^k)|\tau^2\mp 4p^k|^{1/2}\cO(\gamma, f_{\infty}),
\end{equation}
where the sum goes over the choices of $\pm$ and $\tau$ that give irreducible characteristic polynomials $X^2 - \tau X \pm p^k$.
Altu\u{g} then proceeds to prove this can indeed be modified into an appropriate sum over complete lattices of Schwartz functions to which one can perform Poisson Summation. In order to do so, what remains is to smooth the singularities of the orbital integral and fix the conditional convergence introduce by taking the limit at $s = 1$ above. He achieves this by invoking the approximate functional equation.

Recall that the approximate functional equation is a tool to understand the behaviour of $L-$functions in the critical strip, where the Dirichlet Series expansion is not valid. In order to carry this out, the classical functional equation for the $L$ function must already be known for its completion $\Lambda$. This is an instance of corollary \ref{functionaleqn} which we proved in section \ref{Generalization of Zagier's Zeta Function}.

Once the approximate functional equation is performed, the previous sum \eqref{alimanipulation} takes the form
\begin{equation}\label{eqn: incomplete sum 1}
    \displaystyle\sum_{\pm}\displaystyle\sum_{\tau^2\pm 4p^k\neq\square} \displaystyle\sum_{d^2|\tau^2 \pm 4p^k}^{'}\displaystyle\sum_{a=1}^{\infty}G(\pm, \tau, a, d),
\end{equation}
for some complicated function $G$ that depends on the parameters $\pm, \tau, a, d$. The sum over $a$ is introduced by the approximate functional equation. Notice that the terms that index the sum are all terms we have found before: first, $\tau^2\pm 4p^k\neq\square$ means this is not square, but it is a discriminant. The third sum is the one of Zagier's zeta function and includes the congruence conditions. 

The function $G(\pm, \tau, a, d)$, complicated as it might be, makes sense for all choices of $\pm, \tau, a, d$. Due to this fact, we can define the complementary sum
\begin{equation}\label{eqn: incomplete sum 2}
    \displaystyle\sum_{\pm}\displaystyle\sum_{\tau^2\pm 4p^k=\square} \displaystyle\sum_{d^2|\tau^2 \pm 4p^k}^{'}\displaystyle\sum_{a=1}^{\infty}G(\pm, \tau, a, d),
\end{equation}
The congruence conditions, denoted by $'$, mean exactly the same. That is,
\begin{enumerate}
    \item $d^2| \tau^2 \pm 4p^k$,
    \item $\dfrac{\tau^2\pm 4p^k}{4}\equiv 0, 1\pmod{4}$.
\end{enumerate}
Notice how these make sense even if $\tau^2\pm 4p^k$ is a square. In such case we could not define $S_{\gamma}$ because there is no associated regular elliptic matrix nor quadratic extension. This shows the importance of the congruence conditions and their equivalence with the divisibility condition $d|S_{\gamma}$. Without them, we do not know what should the complementary sum be. Thus, we wouldn't be able to complete a lattice nor perform Poisson summation.

In this way, we can add these two sums to \textit{complete} the lattice of integers $\tau$ to get the completed sum
\begin{equation}\label{eqn: complete sum 1}
    \displaystyle\sum_{\pm}\displaystyle\sum_{\tau\in\mathbb{Z}} \displaystyle\sum_{d^2|\tau^2 \pm 4p^k}^{'}\displaystyle\sum_{a=1}^{\infty}G(\pm, \tau, a, d).
\end{equation}
Rearranging the order of summation one gets
\begin{equation}\label{eqn: complete sum 2}
    \displaystyle\sum_{\pm}\displaystyle\sum_{d=1}^{\infty} \displaystyle\sum_{a=1}^{\infty}\displaystyle\sum_{d^2|\tau^2 \pm 4p^k}^{'}G(\pm, \tau, a, d).
\end{equation}
At this stage one studies the inner sum and proves that Poisson Summation can be performed on it. What follows after this is a long analysis of the main term of the spectral side of the Poisson Summation to find in there the contribution of the trivial representation.

The congruence conditions are important once more after Poisson summation as follows. Once it has been performed, it is necessary to simplify a Dirichlet Series that naturally stems from it, and whose coefficients are Kloosterman Sums associated to the character $\chi_{\gamma}$.  Concretely, it is important to evaluate exactly
\begin{equation*}
    \displaystyle\sum_{d=1}^{\infty}\dfrac{1}{d^{2s+1}}\displaystyle\sum_{a=1}^{\infty}\dfrac{K_{a, d}}{a^{s+1}},
\end{equation*}
where
\begin{equation*}
    K_{a, d} = \displaystyle\sum_{m \pmod{4ad^2}}'\left(\dfrac{(m^2\pm 4p^k/d^2)}{a}\right).
\end{equation*}
Here $'$ means once more the congruence conditions for a set of representatives of $\mathbb{Z}$ modulo $4ad^2$. That is, the sum runs over those classes $m$ for which $d$ satisfies the congruence conditions for $m^2 \pm 4p^k$. It turns out that these sums can be evaluated precisely by local methods because the congruence conditions, and the Kloosterman sums they define as above, are multiplicative. That is, one can decompose the above sum into an Euler product and evaluate each factor independently. It is unclear how would one proceed here if one needed the condition $d|S_{\gamma}$ instead. It is even unclear what the definition of the Kloosterman sum would be.

In these computations, at the prime $p$, the Euler factor becomes the contribution, from the finite places, of the trace of the trivial representation (see \cite{LanBE04}, section 2.1).

To appreciate the congruence conditions, one could remove the $'$ in the above Dirichlet Series and simplify. The new Dirichlet series still factors into an Euler product. For odd primes there is no change in them, but for $p=2$ one instead gets that the Euler factor when evaluated at $s=1$ either does not simplify to a simple fraction or it simplifies to a 
\begin{equation*}
    \dfrac{1 - 2^{-(k + 3)}}{1 - 2^{-1}}.
\end{equation*}
This is \textit{not} what comes from the trivial representation.

What we develop in this paper allows us to repeat  for general number fields the steps of the process in which the congruence conditions were relevant. This is being done in \cite{BEgenfields}. All the stages where the congruence conditions were the only obstacle to generalize, as discussed above for $\mathbb{Q}$, can be carried out now with this generalization for $K$.

\section{Questions and future work}\label{Questions and future work}

Our work in this paper highlights complications that will likely appear when generalizing to other contexts of importance. Furthermore, we see that some of these problems can be studied independently of their relationships to the strategy of BE. 

\underline{\textit{1. What regularity do the polynomials of local orders have?}} One of the main observations of our previous discussion is that all of the calculations are possible because of veracity of Arthur's conjecture \ref{arthurconjecture} above. In particular, because we know the explicit left hand side of the equation.

Thanks to the work of Yun in \cite{ZYun} we have, at least for $GL(n, K)$, the existence of the polynomials in the right hand side. However, we do not know them explicitly for $n > 2$ as we do for $GL(2)$. It is of relevance then to know these values, or at least enough of their properties to deduce similar results as the ones used in the case of $GL(2)$. 

Significantly, we see that in the case the local polynomials are produced from a regular elliptic matrix, the Hecke character of the global quadratic allows us to describe uniformly these polynomials which otherwise distribute themselves different families according to the reduced algebra one is working with. 

We could repeat the construction with regular elliptic matrices in higher rank. The nature of the local orders is related to the splitting of primes in the associated global extension. Since we know reciprocity laws are more complicated in higher degree, as well as less explicit, is there such a similar description possible in higher rank?

\underline{\textit{2. What are the concrete congruence conditions?}} As we have pointed out above, the zeta functions of orders exist for $GL(n)$. Thus their product is a candidate to play the role that Zagier's Zeta Function plays here. However, in the process of constructing the global product, what appears naturally as the indexing set is a divisibility condition $I|S_{\gamma}$. If what we have done for $GL(2)$ is to serve as our guide, it should be expected that these conditions are equivalent to other more explicit conditions which allow us to complete the lattice for Poisson Summation and to split the Kloosterman Sums into local Euler Factors. What are these equivalent conditions? 

\underline{\textit{3. The congruence conditions and the nontempered spectrum?}} As we have mentioned, the congruence conditions allow for the evaluation of the Kloosterman sums and lead to the local trace of the trivial representation to be found. We have seen that Zagier´s zeta function satisfies two properties. First, its Euler factors are related to the orbital integrals from the geometric side of the trace formula. Second, in its global expression, it is indexed with a congruence condition that precisely allows the detection of information from the spectral side of the trace formula. How general is this and is there a deeper reason for it? The author does not know a satisfactory answer to this question.

\underline{\textit{4. What applications can the generalizations of Zagier's zeta function have?}} 
Over the rational numbers this function has been used by Zagier  to construct concrete modular forms (see \cite[proposition 3.(iii), section 4, page 130]{zagier} ) associated to quadratic fields. It has also been used by Bykovskii in \cite[equation 1.5]{bykovskii} to settle a conjecture of Iwaniec on the asymptotics of the distribution of prime geodesics related to Selberg's zeta function. It is also used, in the same paper, to study the asymptotics of the class numbers of imaginary quadratic fields as well as on the number of representations of a number as a sum of three squares. More recently it has been used by Soundararajan and Young to study the prime geodesic theorem and its asymptotics (\cite[equations (4) and (5)]{SoundararajanYoung+2013+105+120}). It is an interesting question if our generalization has similar applications. The author hopes to return to this question at a later date.

\bibliographystyle{acm}
\bibliography{Bibliography}

\end{document}